\documentclass[11pt]{article}

\usepackage[utf8]{inputenc}
\usepackage[a4paper,margin=1.1in]{geometry}
\usepackage{amsmath,amssymb,amsfonts,amsthm,mathtools,bm}
\usepackage{longtable}
\usepackage{microtype}
\usepackage[numbers,sort&compress]{natbib}
\usepackage{xcolor}
\usepackage[colorlinks=true,linkcolor=blue!70!black,citecolor=green!50!black,urlcolor=blue!80!black]{hyperref}
\usepackage{doi}
\usepackage{enumitem}
\newtheorem{theorem}{Theorem}[section]
\newtheorem{lemma}[theorem]{Lemma}
\newtheorem{proposition}[theorem]{Proposition}
\newtheorem{corollary}[theorem]{Corollary}
\theoremstyle{definition}
\newtheorem{assumption}[theorem]{Assumption}
\theoremstyle{remark}
\newtheorem{remark}[theorem]{Remark}
\newtheorem*{theorem*}{Theorem} 

\DeclareMathOperator{\Var}{Var}     
\DeclareMathOperator{\Cov}{Cov}

\newcommand{\E}{\mathbb{E}}
\newcommand{\R}{\mathbb{R}}
\newcommand{\dd}{\mathrm{d}}

\newcommand{\avg}[1]{\left\langle #1 \right\rangle}

\numberwithin{equation}{section}

\theoremstyle{remark}

\title{Critical Free-Energy Variance in the Sherrington--Kirkpatrick Model}
\author{Miguel Tierz\\[4pt]
{\normalsize\itshape Shanghai Institute for Mathematics and Interdisciplinary Sciences (SIMIS),}\\
{\normalsize\itshape Block A, No.\ 657 Songhu Road, Yangpu District, Shanghai, China}\\[2pt]
{\normalsize\ttfamily tierz@simis.cn}}
\date{}

\hypersetup{pdftitle={Critical Free-Energy Variance in the Sherrington-Kirkpatrick Model},
            pdfauthor={Miguel Tierz},
            pdfkeywords={Sherrington--Kirkpatrick model, free-energy variance, spin-glass susceptibility, Ornstein--Uhlenbeck interpolation},
            pdflang={en}}

\begin{document}
\setlength{\emergencystretch}{2em}
\maketitle

\begin{abstract}
We study the variance of the free energy in the Gaussian
Sherrington--Kirkpatrick model at zero external field. For every system
size and inverse temperature $\beta>0$, we prove
\[
  2N^2\Var F_N^{\rm off}(\beta)
  \le \log\chi+(\beta^2-1)(\chi-1),
  \qquad \chi=N\E\langle R_{1,2}^2\rangle_\beta,
\]
where $R_{1,2}$ is the replica overlap and
$F_N^{\rm off}=N^{-1}\log Z_N^{\rm off}$ omits diagonal disorder.
The proof uses an exact identity for the edges incident to one spin,
Schur complements, and Gaussian disorder interpolation.
At criticality the bound reduces to $2N^2\Var F_N^{\rm off}(1)\le\log\chi$.
An explicit nonnegative integral measures its gap. A second decomposition
relates this gap to a mixed cubic overlap functional. Existing moment
and cavity estimates imply that its required integral is bounded.
Critical susceptibility asymptotics give the sharp logarithmic upper
coefficient and show that the gap converges exactly when the variance
has a constant term.
We extend an entropy bound throughout bounded critical temperature
windows and establish overlap moment stability on a smaller interval.
\end{abstract}



\section{Introduction}\label{sec:intro}

The Sherrington--Kirkpatrick (SK) model~\cite{SherringtonKirkpatrick1975} is a mean-field spin-glass model. Its thermodynamic limit was established in~\cite{GuerraToninelli2002ThermoLimit}, and the limiting free energy is described by the Parisi formula~\cite{Parisi1979Infinite,Parisi1980OrderParameter,Guerra2003,Talagrand2006Parisi}. For rigorous background on the SK model and mean-field spin glasses, see the monographs~\cite{Talagrand2011Vol1,Panchenko2013Book}. At high temperature $\beta<1$, the unnormalized free energy has $O(1)$ Gaussian fluctuations---equivalently, $F_N$ fluctuates on the $N^{-1}$ scale---by the results in~\cite{AizenmanLebowitzRuelle1987,CometsNeveu1995,GuerraToninelli2002CLT}. At the exact critical inverse temperature $\beta_c=1$ and zero external field $h=0$, the replica symmetric solution is marginally stable in the Almeida--Thouless analysis~\cite{AlmeidaThouless1978,MezardParisiVirasoro1987}. This is the critical regime studied in the fluctuation results~\cite{ChenLam2019,CollinsWoodfinLe2025,DeyKang2026}.

We fix the following basic spin-glass notation. See, for example, \cite[Ch.~1]{Talagrand2011Vol1} and \cite[Chs.~1--2]{Panchenko2013Book}. For configurations $\sigma,\tau\in\{-1,+1\}^N$,
\[
  R_{\sigma,\tau}:=\frac1N\sum_{i=1}^N\sigma_i\tau_i
\]
is their overlap. Replica indices label independent samples from the relevant Gibbs measure, so $R_{a,b}$ means $R_{\sigma^{(a)},\sigma^{(b)}}$. The symbol $\E$ denotes expectation over the Gaussian disorder, while $\langle\cdot\rangle$ denotes Gibbs expectation for a fixed disorder realization. The finite-size spin-glass susceptibility at criticality is
\[
  \chi_{\rm SG}(N):=N\,\E\langle R_{1,3}^2\rangle,
\]
where replicas $1$ and $3$ are independent samples from one critical SK Gibbs measure. When constants are indexed, as in $O_c(1)$, the constant may depend on the fixed parameter $c$ but not on $N$. $A\lesssim B$ means $A=O(B)$, and $A\asymp B$ means both $A\lesssim B$ and $B\lesssim A$.

Throughout this paper $F_N=N^{-1}\log Z_N$ is the normalized free energy. Some cited fluctuation papers instead write $F_N$ for the unnormalized log-partition function $\mathcal F_N=\log Z_N$. Their variances are therefore $N^2$ times the variance convention used here.
Our partition function uses counting measure on $\{\pm1\}^N$.
For comparison with probability spin measures, the precise conversion is
\begin{equation}\label{eq:probability-normalization}
  \mathcal F_N^{\rm prob}(\beta)
  :=\log\bigl(2^{-N}Z_N(\beta)\bigr)
  =NF_N(\beta)-N\log2.
\end{equation}
Thus the choice of spin measure affects the mean through $N\log2$ and
leaves the unnormalized variance unchanged.
In our normalization, the critical finite-size prediction, summarized in~\cite{CollinsWoodfinLe2025} and now proved, is
\[
  \Var F_N(1)=\frac16\frac{\log N}{N^2}+O(N^{-2}).
\]
The corresponding critical Gaussian fluctuation scale for the two-spin spherical SK model was proved in~\cite{Landon2022SphericalCritical}. For the connection between bond chaos and free-energy fluctuations in SK, see~\cite{Aspelmeier2008PRL,Aspelmeier2008JStatMech}. An earlier general upper bound was proved in~\cite{ChenLam2019}, namely $\Var F_N(1)\le C(\log N)^2/N^2$ for a constant $C$ independent of $N$. A variance formula and Gaussian central limit theorem for the unnormalized off-diagonal log-partition function were proved in~\cite{DeyKang2026} for sequences $\beta_N\to1^-$ satisfying
$N^{1/3}(1-\beta_N^2)\to c\in(0,\infty]$.  In particular, when
$N^{1/3}(1-\beta_N^2)$ converges to a positive finite constant, the variance is $\frac16\log N+O(1)$, or $\frac16(\log N)/N^2+O(N^{-2})$ in the normalized convention used here.

At exact criticality, the following results establish the logarithmic scale
and the sharp asymptotics. Both papers work with the
unnormalized log-partition function---Du and Huang in the Gaussian orthogonal ensemble (GOE) convention
adopted here, Schertzer in the off-diagonal convention---so their statements
are quoted below in the normalized convention, the diagonal conventions
differing only by the explicit $O(N^{-2})$ shift of
Lemma~\ref{lem:GOE-vs-offdiag}.
Schertzer~\cite{Schertzer2026CriticalSK} proved, by elementary arguments, the
unconditional bounds
\[
  \Var F_N(1)\le\frac14\frac{\log N}{N^2}+O(N^{-2}),
  \qquad
  \Var F_N(1)\ge\frac{\log\log\log N}{2N^2}-O(N^{-2}),
\]
the second showing that $N^2\Var F_N(1)$ diverges.  The upper bound rests on a second-moment inequality obtained from Gaussian convexity, $\Var(\log Z_N)\le\log\bigl(\E Z_N^2/(\E Z_N)^2\bigr)$, whose
right-hand side grows at criticality as the Curie--Weiss scale
$\tfrac14\log N+O(1)$. The lower bound sums simple disorder cycles in
the Hermite expansion of the variance~\cite{Chatterjee2009}, with
errors controlled by the classical critical overlap bound
$\E\langle R_{1,3}^2\rangle=O(N^{-1/2})$~\cite[Proposition~11.7.8]{Talagrand2011Vol2}.
Du and Huang~\cite{DuHuang2026} then proved the full prediction,
\[
  \Var F_N(1)=\frac16\frac{\log N}{N^2}+O(N^{-2}),
\]
together with a Gaussian central limit theorem and the critical overlap
scaling $\E\langle R_{1,3}^2\rangle\asymp N^{-2/3}$---that is,
$\chi_{\rm SG}(N)\asymp N^{1/3}$---conjectured in~\cite[Research Problem~11.7.4]{Talagrand2011Vol2}. The key input
is an $L^2$ comparison between the Ising and spherical SK partition
functions.  The present paper instead develops an elementary finite-$N$
identity at exact criticality.  It yields an upper bound in terms of susceptibility and identifies an explicit signed remainder governing saturation.
The critical overlap input is stated separately. Proposition~\ref{prop:cubic-from-moments}
verifies the integrated remainder condition directly from existing moment
and cavity estimates, without using the final variance theorem.

A subsequent proof of the critical variance asymptotics uses entropy,
Gaussian convexity, and a comparison with Gaussian synchronization
\cite[Theorem~3]{YimingChen2026}. A further development identifies the
limiting critical overlap distribution and proves
\begin{equation}\label{eq:chi-limit}
  N^{-1/3}\chi_{\rm SG}(N)\longrightarrow C_\chi\in(0,\infty)
\end{equation}
\cite[Theorem~1.1 and Corollary~1.2]{DuHuang2026Overlap}.
Here $C_\chi$ is the expected second moment of the random limiting
overlap measure defined in that work. Our focus is the finite-$N$
relation between variance and susceptibility. In the off-diagonal
convention, the bound has the exact form
$2N^2\Var F_N^{\rm off}(1)\le\log\chi_{\rm SG}(N)$.
Theorems~\ref{thm:intro-SK-upper} and~\ref{prop:finite-beta-bound}
state the versions at criticality and at general inverse temperature.
Proposition~\ref{prop:exact-critical-gap} identifies the integral gap.
A companion paper in preparation~\cite{Tierz2026Constants} develops
the spectral analysis of the critical spherical model and its comparison
with the Ising model. When the finite part of the critical variance
converges, we write
\begin{equation}\label{eq:intro-variance-constant}
  N^2\Var F_N^{\rm off}(1)
  =\frac16\log N+c_{\rm Ising}+o(1).
\end{equation}
Using~\eqref{eq:chi-limit} and the integrated remainder bound of
Proposition~\ref{prop:cubic-from-moments},
Proposition~\ref{prop:limiting-critical-gap} below proves that the nonnegative integral gap is bounded and converges
if and only if the limit~\eqref{eq:intro-variance-constant} exists, in
which case its limit is $\log C_\chi-2c_{\rm Ising}$. This identifies
the finite difference between the susceptibility bound and the variance.
The existence proof and the numerical evaluation of the constant are
deferred to that paper and are not used here. The present paper retains
the finite-size inequalities and their overlap interpretation.

A static Gaussian Poincar\'e estimate already relates the free-energy
variance to the overlap susceptibility.  Writing
\[
  \chi_{\rm SG}(N;\beta)
  :=N\,\E\langle R_{1,3}^2\rangle_\beta,
  \qquad
  \chi_{\rm SG}(N;1)=\chi_{\rm SG}(N),
\]
it gives
\begin{equation}\label{eq:intro-poincare}
  \Var F_N(\beta)
  \le \frac{\beta^2}{2N}\,\E\langle R_{1,3}^2\rangle_\beta
  =\frac{\beta^2}{2N^2}\,\chi_{\rm SG}(N;\beta),
\end{equation}
where replicas $1$ and $3$ are independent samples from the same Gibbs measure. This is the Gaussian Poincar\'e inequality applied to the disorder variables, with the SK covariance normalization $\xi(r)=r^2/2$~\cite{BakryGentilLedoux2014,Chatterjee2014,Panchenko2013Book}. The scale $\chi_{\rm SG}(N)\asymp N^{1/3}$, long associated with the critical SK coefficient in the fluctuation and finite-size scaling literature~\cite{AspelmeierKatzgraberLarsonMooreWittmannYeo2016,CollinsWoodfinLe2025,DeyKang2026}, is now a theorem~\cite{DuHuang2026}. At that scale, \eqref{eq:intro-poincare} is much larger than $(\log N)/N^2$. The Ornstein--Uhlenbeck (OU) identity used below is therefore not only a variance bound. It records how the overlap response decays as the two disorders are decorrelated. A static comparison between overlap and variance, integrated over $\beta$ and based on the annealed--quenched gap, appears in~\cite[Proposition~3]{Schertzer2026CriticalSK}.

Our result has two parts. First, we prove an exact finite-$N$ incident-edge identity via Price's theorem~\cite{Price1958} whose signed remainder structure gives an unconditional one-sided theorem:
\[
  \Var F_N(1)
  \le
  \frac{1+\log\chi_{\rm SG}(N)}{2N^2},
  \qquad
  \chi_{\rm SG}(N):=N\,\E\langle R_{1,3}^2\rangle.
\]
This converts any susceptibility bound $\chi_{\rm SG}(N)\le N^{\gamma+o(1)}$ into the upper-bound coefficient $\gamma/2$.  The self-contained endpoint input $\chi_{\rm SG}(N)\le(N+1)/2$ of Remark~\ref{rem:endpoint-sanity} already improves the logarithmic power in the bound~\cite{ChenLam2019} from $(\log N)^2$ to $\log N$. The classical critical overlap bound $\chi_{\rm SG}(N)=O(N^{1/2})$~\cite[Proposition~11.7.8]{Talagrand2011Vol2} gives the coefficient $1/4$, and the overlap scaling in~\cite{DuHuang2026} gives the sharp coefficient $1/6$ (Theorem~\ref{thm:intro-SK-upper} and Remark~\ref{rem:dictionary-instances}), without using the fluctuation analysis in that paper. Second, the signed remainders characterize the gap in this inequality.
Their integrated bound follows from the moment and cavity estimates
of~\cite[Propositions~5.1 and~5.7]{DuHuang2026}, applied directly to the
exact identity in Proposition~\ref{prop:cubic-from-moments}. The same
estimates also bound the mixed cubic functional. This supplies a
verification within the remainder decomposition using explicitly cited
critical inputs.

\subsection{Main results}\label{sec:intro-main}

We work with the GOE normalization
\[
  H_N(\sigma)=\frac{1}{2\sqrt N}\sum_{1\le i,j\le N}g_{ij}\sigma_i\sigma_j,
  \qquad \sigma_i\in\{\pm1\},
\]
where $g=(g_{ij})$ is a standard GOE matrix, normalized so that
$g_{ij}=g_{ji}\sim N(0,1)$ for $i<j$ and $g_{ii}\sim N(0,2)$.  With
$R_{\sigma,\tau}=N^{-1}\sum_i\sigma_i\tau_i$, this gives
\[
  \Cov(H_N(\sigma),H_N(\tau))=\frac N2 R_{\sigma,\tau}^2.
\]

\paragraph{Disorder convention.}
The diagonal terms in the GOE convention are spin-independent and produce an
explicit $O(N^{-2})$ variance shift relative to the off-diagonal convention.

\begin{lemma}[GOE vs. off-diagonal normalization]\label{lem:GOE-vs-offdiag}
Let
\[
  H_N^{\mathrm{GOE}}(\sigma):=\frac{1}{2\sqrt{N}}\sum_{i,j=1}^N g_{ij}\,\sigma_i\sigma_j
\]
with $(g_{ij})$ a standard GOE matrix, and let
\[
  H_N^{\mathrm{off}}(\sigma):=\frac{1}{\sqrt{N}}\sum_{1\le i<j\le N} g_{ij}\,\sigma_i\sigma_j.
\]
Then
\[
  H_N^{\mathrm{GOE}}(\sigma)=H_N^{\mathrm{off}}(\sigma)+D_N,
  \qquad D_N:=\frac{1}{2\sqrt N}\sum_{i=1}^N g_{ii},
\]
and $D_N$ is spin-independent with $\Var(D_N)=1/2$. Consequently,
\[
  \Var(F_N^{\mathrm{GOE}}(\beta))
  =\Var(F_N^{\mathrm{off}}(\beta))+\frac{\beta^2}{2N^2}.
\]
\end{lemma}

\begin{proof}
Split the GOE sum into off-diagonal and diagonal parts. Since $\sigma_i^2=1$,
\[
  H_N^{\mathrm{GOE}}(\sigma)=H_N^{\mathrm{off}}(\sigma)+\frac{1}{2\sqrt N}\sum_i g_{ii}.
\]
The diagonal term is independent of the off-diagonal disorder and has variance
$N\cdot2/(4N)=1/2$.  The variance identity follows from
$Z_N^{\mathrm{GOE}}(\beta)=e^{\beta D_N}Z_N^{\mathrm{off}}(\beta)$.
\end{proof}

In particular, the GOE and off-diagonal conventions induce exactly the same
Gibbs measures.  All Gibbs expectations, overlap quantities,
$\chi_N(t)$, and $\chi_{\rm SG}(N)$ are therefore identical. Only the
free-energy variance changes by the explicit shift in
Lemma~\ref{lem:GOE-vs-offdiag}.  Consequently all leading statements below
also hold in the off-diagonal convention.

Let
\[
  F_N(\beta):=\frac1N\log Z_N(\beta),
  \qquad Z_N(\beta):=\sum_\sigma e^{\beta H_N(\sigma)}.
\]
For $t\ge0$ set $\rho_t=e^{-t}$ and let
\[
  H_{N,t}^{(a)}:=\sqrt{\rho_t}\,H_N+\sqrt{1-\rho_t}\,\widetilde H_N^{(a)},
  \qquad a=1,2,
\]
where $H_N,\widetilde H_N^{(1)},\widetilde H_N^{(2)}$ are independent.
The notation $\langle\cdot\rangle_t$ denotes the product Gibbs expectation
at $\beta=1$.  Replicas $1,3$ are sampled from the first copy and replicas
$2,4$ from the second copy, unless otherwise stated.  Define
\[
  \chi_N(t):=N\,\E\langle R_{1,2}^2\rangle_t,
  \qquad
  \chi_{\rm SG}(N):=N\,\E\langle R_{1,3}^2\rangle=\chi_N(0).
\]

\begin{theorem}[Finite-size susceptibility bound]\label{thm:intro-SK-upper}
For the SK model at $(\beta,h)=(1,0)$ and every $N\ge1$,
\begin{equation}\label{eq:intro-exact-upper}
  2N^2\Var F_N(1)\le1+\log\chi_{\rm SG}(N),
  \qquad
  2N^2\Var F_N^{\rm off}(1)\le\log\chi_{\rm SG}(N).
\end{equation}
Thus $\chi_{\rm SG}(N)\le N^{\gamma+o(1)}$ gives the upper
coefficient $\gamma/2$. In particular, the critical susceptibility
asymptotics~\eqref{eq:chi-limit} imply
\begin{equation}\label{eq:intro-constant-upper}
  \limsup_{N\to\infty}\left(
    N^2\Var F_N^{\rm off}(1)-\frac16\log N\right)
  \le\frac12\log C_\chi.
\end{equation}
In the GOE convention the right-hand side is
$(1+\log C_\chi)/2$.
\end{theorem}

The proof of~\eqref{eq:intro-exact-upper} uses no asymptotic overlap
input. The last assertion uses only~\eqref{eq:chi-limit}, not the
critical variance theorem. More generally,
Theorem~\ref{prop:finite-beta-bound} proves, for every $\beta>0$,
\begin{equation}\label{eq:intro-all-beta}
  2N^2\Var F_N^{\rm off}(\beta)
  \le\log\chi+(\beta^2-1)(\chi-1),
  \qquad \chi=\chi_{\rm SG}(N;\beta).
\end{equation}

The core input is the exact incident-edge identity obtained from Price's
theorem~\cite{Price1958}.  Put
\[
  \theta_N(t):=1-\rho_t+\frac{2\rho_t}{N}.
\]
The global Schur and interpolation remainders are the normalized sums of the
per-pair quantities in \eqref{eq:R-schur}--\eqref{eq:R-interp}:
\[
  \mathfrak R_N^{\rm Schur}(t)
  :=\frac1N\sum_{j=1}^N\sum_{i\ne j}\mathcal R_{ij}^{\rm Schur}(t),
  \qquad
  \mathfrak R_N^{\rm interp}(t)
  :=\frac1N\sum_{j=1}^N\sum_{i\ne j}\mathcal R_{ij}^{\rm interp}(t).
\]
Proposition~\ref{prop:Qe-direct-ladder} gives
\begin{equation}\label{eq:intro-Qe-exact-identity}
  \theta_N(t)\chi_N(t)
  =1+\frac{\rho_t}{N}
  +\mathfrak R_N^{\rm Schur}(t)
  +\mathfrak R_N^{\rm interp}(t),
\end{equation}
and Lemma~\ref{lem:remainder-signs} gives
$\mathfrak R_N^{\rm Schur}(t)\le0$ and
$\mathfrak R_N^{\rm interp}(t)\le0$.

\begin{assumption}[Integrated remainder condition]\label{ass:SK-remainder-closure}
For $c>1$ define
\begin{equation}\label{eq:def-depletion-remainder}
  \mathcal D_N^{(c)}
  :=\int_{c/\chi_{\rm SG}(N)}^1 e^{-t}
  \frac{-\mathfrak R_N^{\rm Schur}(t)-\mathfrak R_N^{\rm interp}(t)}
       {\theta_N(t)}\,\dd t,
\end{equation}
We regard $[c/\chi_{\rm SG}(N),1]$ as empty when
$c/\chi_{\rm SG}(N)\ge1$, and every integral over that window is
defined to be zero.  The condition used in the saturation criterion is that, for one fixed $c>1$---or,
equivalently by Lemma~\ref{lem:cutoff-stability}, for every fixed $c>1$---
\begin{equation}\label{eq:ass-SK-integrated-rem}
  \mathcal D_N^{(c)}=O_c(1).
\end{equation}
\end{assumption}

\begin{theorem}[Integrated signed-remainder decomposition]\label{thm:intro-remainder}
For every fixed $c>1$,
\[
  2N^2\Var F_N(1)
  =\log\chi_{\rm SG}(N)-\mathcal D_N^{(c)}+O_c(1),
  \qquad \mathcal D_N^{(c)}\ge0.
\]
This is proved in Proposition~\ref{prop:saturation-remainder}.
\end{theorem}

\begin{theorem}[Saturation criterion]\label{thm:intro-SK}
Under Assumption~\ref{ass:SK-remainder-closure},
\[
  \Var F_N(1)
  =\frac{\log\chi_{\rm SG}(N)}{2N^2}+O(N^{-2}).
\]
\end{theorem}

\begin{corollary}[Upper coefficient from susceptibility scaling]
\label{cor:intro-SK-upper-1-6}
If $\chi_{\rm SG}(N)\le N^{\gamma+o(1)}$ for some fixed $\gamma\in(0,1]$, then
\[
  \Var F_N(1)
  \le\left(\frac{\gamma}{2}+o(1)\right)\frac{\log N}{N^2}.
\]
No signed-remainder assumption is needed for this upper bound.
\end{corollary}

\begin{remark}[Consequences of known susceptibility bounds]\label{rem:dictionary-instances}
The self-contained value $\gamma=1$ gives the coefficient $1/2$.  The value
$\gamma=1/2$ holds unconditionally by the classical critical overlap bound $\E\langle R_{1,3}^2\rangle=O(N^{-1/2})$~\cite[Proposition~11.7.8]{Talagrand2011Vol2}
(recorded in this form in~\cite[Remark~2.2]{ChenLam2019}), so
Theorem~\ref{thm:intro-SK-upper} yields
$\Var F_N(1)\le\frac14(\log N)/N^2+O(N^{-2})$ from classical inputs alone.
The same coefficient was proved in~\cite{Schertzer2026CriticalSK}
using Gaussian convexity and the second-moment bound
$\Var(\log Z_N)\le\log(\E Z_N^2/(\E Z_N)^2)$.
At criticality the ratio $\E Z_N^2/(\E Z_N)^2$ has order $N^{1/4}$,
so its logarithm is $\tfrac14\log N+O(1)$.
This growth follows from the Curie--Weiss quartic scale.  The
critical value $\gamma=1/3$ follows from~\cite{DuHuang2026}, where the bound $\E\langle R_{1,3}^2\rangle\asymp N^{-2/3}$ has constants independent of $N$, so the error improves to
$O(N^{-2})$, as recorded in the last display of
Theorem~\ref{thm:intro-SK-upper}.
\end{remark}

\begin{corollary}[$1/6$ equality from bounded remainder and susceptibility scaling]\label{cor:intro-SK-1-6}
If Assumption~\ref{ass:SK-remainder-closure} holds and
$\chi_{\rm SG}(N)=N^{1/3+o(1)}$, then
\[
  \Var F_N(1)
  =\frac16\frac{\log N}{N^2}
   +o\!\left(\frac{\log N}{N^2}\right).
\]
\end{corollary}

\begin{remark}[Verification from the critical overlap estimates]
\label{rem:assumption-verified}
Proposition~\ref{prop:cubic-from-moments} verifies
Assumption~\ref{ass:SK-remainder-closure} directly from
\cite[Propositions~5.1 and~5.7]{DuHuang2026}.
The latter proposition records cavity estimates from~\cite{DeyKang2026}.
The same estimates also prove \eqref{eq:integrated-cubic-control} for
the mixed cubic functional. The resulting saturation statement depends
on those critical inputs, while the upper bound
\eqref{eq:intro-exact-upper} does not.
The established variance theorem also implies this bound a posteriori,
but is unnecessary for the verification just described. The stronger
pointwise condition \eqref{eq:mixed-cubic-control} is not proved here.
A verification of Assumption~\ref{ass:SK-remainder-closure} within the
signed-remainder framework itself, without the cavity expansion
behind~\cite[Proposition~5.7]{DuHuang2026}, remains open.
\end{remark}

\begin{remark}[Finite-size diagnostic for the integrated remainder]\label{rem:finite-size-diagnostic}
Let
$\mathfrak R_N(t)=\mathfrak R_N^{\rm Schur}(t)+\mathfrak R_N^{\rm interp}(t)$.
The exact identity gives
\[
  \mathfrak R_N(t)=\theta_N(t)\chi_N(t)-1-\frac{\rho_t}{N}.
\]
Thus $\mathcal D_N^{(c)}$ can be reconstructed from the
 disorder-averaged curve $\chi_N(t)$ without estimating the Schur and
interpolation pieces separately.  Exact enumeration at very small $N$, or
Monte Carlo sampling at moderate $N$, therefore needs only the OU
susceptibility curve and the deterministic cutoff $c/\chi_{\rm SG}(N)$.
Such computations are consistency checks rather than asymptotic evidence.
\end{remark}

\begin{remark}[Integral interpretation]\label{rem:integrated-resolvent}
The exact identity gives
\[
  \mathcal D_N^{(c)}
  =\int_{c/\chi_{\rm SG}(N)}^1 e^{-t}
  \left[\frac{1+\rho_t/N}{\theta_N(t)}-\chi_N(t)\right]\dd t.
\]
Thus its boundedness is an integrated lower estimate on the coupled
susceptibility relative to the rational upper profile.
The quantitative estimate from~\cite[Proposition~5.7]{DuHuang2026},
used in Proposition~\ref{prop:cubic-from-moments}, controls this integral
directly through the exact identity.
The qualitative statement $(1-\rho_t)\chi_N(t)=1+o(1)$ away from the
critical scale would not by itself give the required bounded error.
\end{remark}

\begin{remark}[Reduction to the Schur remainder]\label{rem:narrow-assumption}
Lemma~\ref{lem:interp-integrated-app} bounds the interpolation contribution by
$O(\chi_{\rm SG}(N)^2/N)$.  By the critical scaling
$\chi_{\rm SG}(N)\asymp N^{1/3}$ of~\cite{DuHuang2026}, this contribution is
$O(N^{-1/3})=o(1)$, so Assumption~\ref{ass:SK-remainder-closure} is
unconditionally equivalent to
\begin{equation}\label{eq:ass-SK-narrow}
  \int_{c/\chi_{\rm SG}(N)}^1 e^{-t}
  \frac{-\mathfrak R_N^{\rm Schur}(t)}{\theta_N(t)}\,\dd t
  =O_c(1).
\end{equation}
\end{remark}

\subsection{Conceptual summary}\label{sec:intro-saturation}

The OU variance representation gives
\[
  \Var F_N(1)=\frac{1}{2N^2}\int_0^\infty e^{-t}\chi_N(t)\,\dd t.
\]
The exact incident-edge identity and the global remainder signs give the
pointwise rational upper bound, while the monotonicity
$\chi_N(t)\le\chi_{\rm SG}(N)$ supplies the finite-size cutoff
$t_\ast\asymp\chi_{\rm SG}(N)^{-1}$.  Proposition~\ref{prop:saturation-remainder}
then combines the integrals below the cutoff, between the cutoff and $1$, and above $1$ into the remainder identity of Theorem~\ref{thm:intro-remainder}.  The unconditional
bound is $\mathcal D_N^{(c)}\ge0$. Saturation is
$\mathcal D_N^{(c)}=O_c(1)$.

\subsection{Relation to four-replica response formulas}

We compare the observable used below with the four-replica response formulas in the literature. In that literature, the replicon sector consists of overlap fluctuations about the replica symmetric state whose symmetric coefficient matrices have zero diagonal and zero row sums~\cite{AlmeidaThouless1978,MezardParisiVirasoro1987,Aspelmeier2008PRL}. The cited interpolation formulas involve two- and four-replica overlaps between systems with correlated disorder~\cite{Aspelmeier2008JStatMech,Aspelmeier2008PRL}.

The present proof uses a different observable. Replica~$1$ is sampled from the first OU copy, while replicas~$2$ and~$4$ are sampled independently from the second OU copy, and
\[
  Q_e=\frac{1}{\sqrt2}(R_{1,2}-R_{1,4}).
\]
At zero field, replica independence and spin-flip symmetry give $N\E\langle Q_e^2\rangle_t=\chi_N(t)$.  Lemma~\ref{lem:Qe-local-vertex} separately
records the two independent nonzero mixed derivatives, $1/\sqrt2$ and $-1/\sqrt2$, whose squares sum to~$1$.
This Hessian calculation is interpretive and is not used in the proof of the
scalar incident-edge identity.  In that proof, the constant term comes from
the exact diagonal normalization $C_{jj}^e=1$ in
Lemma~\ref{lem:cross-only-vertex}, while the finite-size denominator
$1-\rho_t+2\rho_t/N$ comes from the summation over indices distinct from the fixed site in
Proposition~\ref{prop:Qe-direct-ladder}.  The $1/6$ variance coefficient is
therefore not a local response coefficient.  In the conditional theorem it
is the product of the covariance normalization $1/2$ and the susceptibility
exponent $1/3$.

The OU interpolation also describes bond chaos. The heuristic link between free-energy fluctuations and bond chaos~\cite{BouchaudKrzakalaMartin2003} was expressed as an exact relation for SK in~\cite{Aspelmeier2008PRL,Aspelmeier2008JStatMech}. Rigorous disorder-chaos results in related Gaussian settings are discussed in~\cite{Chatterjee2009}. Thus $\chi_N(t)$ measures the overlap between equilibrium states under slightly decorrelated disorder. The cutoff scale $t_\ast\asymp\chi_{\rm SG}(N)^{-1}$ is therefore, by the susceptibility scaling of~\cite{DuHuang2026}, of order $N^{-1/3}$. This disorder scale is related to the critical temperature scale in spherical SK-type models. The normalized GOE eigenvalue fluctuations themselves have scale $N^{-2/3}$~\cite{TracyWidom1996,BaikLee2016,BaikLee2017Ferromagnetic,Landon2022SphericalCritical,CollinsWoodfinLe2025Bipartite,JohnstoneKlochkovOnatskiPavlyshyn2024}.

\subsection{Proof structure}\label{sec:intro-proof-structure}

The proof has the following dependencies.
\begin{enumerate}[leftmargin=2.3em]
  \item Theorem~\ref{thm:var-identity} gives the exact OU variance identity.
  \item Proposition~\ref{prop:Qe-direct-ladder} gives the finite-$N$
  incident-edge identity, and Lemma~\ref{lem:remainder-signs} gives the two
  global remainder signs.
  \item Lemma~\ref{lem:chi-CS} proves that disorder decorrelation decreases
  the coupled susceptibility. Proposition~\ref{prop:exact-critical-gap}
  integrates its two upper bounds and gives the exact gap identity.
  \item Proposition~\ref{prop:saturation-remainder} combines the integrals below the cutoff, between the cutoff and $1$, and above $1$ to retain the separate local
  remainders. Theorem~\ref{thm:critical-variance-constant} follows.
  \item Proposition~\ref{prop:cubic-from-moments} verifies the integrated
  remainder condition directly from existing overlap moment and cavity
  estimates, and also proves the integrated cubic bound.
  \item Appendix~\ref{app:high-temperature} tracks the identity at general
  $\beta$ and recovers the high-temperature bound.
  \item Appendix~\ref{app:temperature-extensions} proves the uniform
  variance extension and the overlap moment stability estimate.
\end{enumerate}

\subsection*{Relation to the established asymptotics}

The critical variance asymptotic and Gaussian limit are established
in~\cite{DuHuang2026}. See also~\cite{YimingChen2026} for the variance.
The independent finite-size content here is the incident-edge identity,
its two remainder signs, and the resulting inequality in terms of
susceptibility at a single temperature. The high-temperature upper bound
also follows from the second-moment argument
of~\cite[Section~2.2]{Schertzer2026CriticalSK}. Here it serves as a
consistency check. The exact integral gap describes the information
lost in the upper bound, and the mixed cubic comparison supplies
additional structure for that gap.

Two further consequences are proved with their external inputs explicit.
Extending the entropy argument of~\cite[Section~3.1]{YimingChen2026}
and using the critical random graph estimate~\cite{JansonSpencer2007},
Proposition~\ref{prop:uniform-window-variance} gives, for every fixed $B$,
\begin{equation}\label{eq:intro-window-variance}
 \sup_{|\beta^2-1|\le BN^{-1/3}}
 N^2\Var F_N^{\rm off}(\beta)\le\frac16\log N+C_B.
\end{equation}
This is an extension of that proof, not an ordinary SK susceptibility
estimate. Separately, Theorem~\ref{thm:moment-stability} combines critical
overlap and free-energy tails with temperature differentiation. With
$m_p(s)=\E\langle|R_{12}|^p\rangle_{\beta=\sqrt{s}}$, it proves
\begin{equation}\label{eq:intro-moment-stability}
 \sup_{|s-1|\le\eta N^{-4/9}}
 N^{p/3}|m_p(s)-m_p(1)|=O_p(N^{-1/9})
 \qquad(p>0).
\end{equation}
In particular, $m_4(s)/m_2(s)^2$ remains bounded on that interval.
The ordinary moment ratio throughout a fixed $N^{-1/3}$ window is
not established by these arguments.

\section{The Ornstein--Uhlenbeck variance identity}\label{sec:setup}

We use an OU deformation of the Hamiltonian to
represent the variance. General Gaussian covariance formulas and replica
interpolation provide related background~\cite{HelfferSjostrand1994,GuerraToninelli2002QRC,Chatterjee2014}. Let $N\ge 1$ be the number of spins, $\Sigma_N=\{-1,+1\}^N$ the set of
Ising configurations, and $\{H_N(\sigma)\}_{\sigma\in\Sigma_N}$ a centered Gaussian process with
\[
\Cov\big(H_N(\sigma),H_N(\sigma')\big)=N\,\xi\big(R_{\sigma,\sigma'}\big),
\]
where $R_{\sigma,\sigma'}$ is the overlap for the model under consideration and $\xi:[-1,1]\to\R$ is of class $C^1$. This is the usual mixed $p$-spin covariance notation~\cite[Chs.~1--2]{Panchenko2013Book}.
Let $Z_N(\beta)=\sum_{\sigma\in\Sigma_N}\exp\big(\beta H_N(\sigma)\big)$ be the partition function,
$F_N(\beta)=\frac{1}{N}\log Z_N(\beta)$ the free energy. For $t\ge 0$ set $\rho_t:=e^{-t}$ and define the
pair of correlated Hamiltonians
$H_t^{(a)}=\sqrt{\rho_t}\,H+\sqrt{1-\rho_t}\,\widetilde H^{(a)}$, $a=1,2$.
For background on Gaussian covariance formulas and quadratic replica
coupling, see~\cite{HelfferSjostrand1994,GuerraToninelli2002QRC}.
Denote by $\langle\cdot\rangle_t$ the Gibbs average under the coupled pair
$\big(H_t^{(1)},H_t^{(2)}\big)$ at inverse temperature $\beta$.
\begin{remark}[Time convention for the OU interpolation]\label{rem:rho-convention}
Throughout the paper the OU parameter is the Hamiltonian correlation
\[
  \rho_t:=e^{-t}.
\]
Thus
\[
\Cov\bigl(H_t^{(1)}(\sigma),H_t^{(2)}(\sigma')\bigr)
=\rho_t\,\Cov\bigl(H(\sigma),H(\sigma')\bigr).
\]
Equivalently, the common disorder enters each Hamiltonian with coefficient
$\sqrt{\rho_t}=e^{-t/2}$. Some references instead parametrize the
common coupling by a number $\hat\rho$ and write
$J^{(a)}_{ij}=\hat\rho g_{ij}+\sqrt{1-\hat\rho^2}\,g^{(a)}_{ij}$. In that convention the Hamiltonian correlation is $\hat\rho^2$.
All formulas in this paper use the Hamiltonian-correlation convention
$\rho_t=e^{-t}$.
\end{remark}

\paragraph{OU-coupled replicas and notation.}
We use the coupled Hamiltonians defined above. The Gaussian covariance
and replica interpolation background is discussed in~\cite{HelfferSjostrand1994,GuerraToninelli2002QRC,Chatterjee2014}. Fix $t\ge 0$ and let $\rho_t=e^{-t}$. For $a\in\{1,2\}$, let
\[
H_t^{(a)} \;=\; \sqrt{\rho_t}\,H \;+\; \sqrt{1-\rho_t}\, \widetilde H^{(a)},
\]
where $\widetilde H^{(1)},\widetilde H^{(2)}$ are i.i.d.\ copies of $H$, independent of $H$.
At inverse temperature $\beta>0$, we consider the product Gibbs measure on pairs
$(\sigma^{(1)},\sigma^{(2)})\in\Sigma_N\times\Sigma_N$ with density proportional to
$\exp\{\beta H_t^{(1)}(\sigma^{(1)})+\beta H_t^{(2)}(\sigma^{(2)})\}$.
We write $\avg{\cdot}_t$ for expectation with respect to this product Gibbs measure and
$\E[\cdot]$ for disorder expectation. Overlaps are denoted by $R_{1,2}=R_{\sigma^{(1)},\sigma^{(2)}}$,
and the covariance structure is $\Cov\!\big(H(\sigma),H(\sigma')\big)=N\,\xi\!\big(R_{\sigma,\sigma'}\big)$.

\paragraph{Big-$O$ conventions.}
Throughout, for quantities depending on $N$ (and possibly on auxiliary parameters such as $c,t,\beta$), we write
$A_N=O(B_N)$ if there exists a constant $C<\infty$, independent of $N$ and uniform over the indicated parameter ranges, such that
$|A_N|\le C\,|B_N|$.
If a subscript is present, e.g.\ $O_c(1)$ or $O_a(1)$, the constant is allowed to depend on the parameter(s) in the subscript but is still uniform in $N$ and in the remaining variables.
We also use $A\lesssim B$ as shorthand for $A=O(B)$, and $A\asymp B$ to mean both $A\lesssim B$ and $B\lesssim A$.

\paragraph{Notation reference.}
The following table lists the objects used in the variance proof. Replicas
$1,3$ are independent samples from the first OU system and replicas $2,4$
are independent samples from the second OU system, unless a different sampling
rule is stated.

{\footnotesize
\renewcommand{\arraystretch}{1.12}
\setlength{\LTleft}{\fill}\setlength{\LTright}{\fill}
\begin{longtable}{p{0.24\linewidth} p{0.66\linewidth}}
\hline
$t$, $\rho_t$ & OU time and Hamiltonian correlation, with $\rho_t=e^{-t}$.\\
$\langle\cdot\rangle_t$ & Gibbs expectation for the OU-coupled pair at time $t$.\\
$R_{a,b}$ & Overlap $N^{-1}\sum_i\sigma_i^{(a)}\sigma_i^{(b)}$.\\
$\chi_N(t)$ & Coupled susceptibility $N\,\E\langle R_{1,2}^2\rangle_t$.\\
$\chi_{\rm SG}(N)$ & Single-system spin-glass susceptibility $N\,\E\langle R_{1,3}^2\rangle$. Equivalently $\chi_N(0)$.\\
$Q_e$ & Cross-system fluctuation $2^{-1/2}(R_{1,2}-R_{1,4})$, used to express $\chi_N(t)$. See Lemma~\ref{lem:cross-only-vertex}.\\
$q_j^e$ & Site-$j$ version of $Q_e$, defined in Lemma~\ref{lem:Qe-local-vertex}.\\
$\theta_N(t)$ & Finite-$N$ denominator $1-\rho_t+2\rho_t/N$ in the incident-edge identity.\\
$\mathfrak R_N^{\rm Schur}(t)$, $\mathfrak R_N^{\rm interp}(t)$ & Normalized sums $N^{-1}\sum_j\sum_{i\ne j}\mathcal R_{ij}$ of the per-pair remainders in~\eqref{eq:R-schur}--\eqref{eq:R-interp}. See~\eqref{eq:global-remainders}.\\
$\E_{j,\lambda}$ & Disorder expectation in which the two disorder copies have cross-covariance $\lambda$ on edges incident to site $j$ and cross-covariance $\rho_t$ on all other edges. See Appendix~\ref{app:replicon}.\\
$\mathcal D_N^{(c)}$ & Integrated signed remainder on the critical window. See \eqref{eq:def-depletion-remainder}.\\
$q_j^{23}$ & Fixed-site product $\sigma_j^{(2)}\sigma_j^{(3)}$ appearing in the Schur cubic term.\\
$\mathfrak C_N(t)$ & Mixed cubic functional in \eqref{eq:mixed-cubic-functional}, used in the Schur estimate.\\
\hline
\end{longtable}
}

\begin{theorem}[Variance identity]\label{thm:var-identity}
For all $\beta>0$,
\begin{equation}\label{eq:var-id-general}
\Var F_N(\beta)\;=\;\frac{\beta^2}{N}\int_0^\infty e^{-t}\;\E\big\langle \xi\!\big(R_{1,2}\big)\big\rangle_t\,\dd t.
\end{equation}
In particular, for SK with $\xi(r)=r^2/2$,
\begin{equation}\label{eq:var-id-SK}
\Var F_N(\beta)\;=\;\frac{\beta^2}{2N}\int_0^\infty e^{-t}\;\E\langle R_{1,2}^2\rangle_t\,\dd t.
\end{equation}
\end{theorem}

For the SK specialization \eqref{eq:var-id-SK}, this is the variance formula~\cite[Theorem~3.8]{Chatterjee2009}, written with the time parametrization $\rho_t=e^{-t}$ of Remark~\ref{rem:rho-convention}. The
short derivation below is included to fix conventions and to cover general
$\xi$.

\begin{proof}[Proof of Theorem~\ref{thm:var-identity}]
Choose underlying independent standard Gaussian coordinates $g=(g_k)$ and a
linear map $A$ such that the Gaussian field is $H=Ag$.  Apply the standard
OU covariance representation to the smooth function
$f(g):=F_N(Ag)$ in these nondegenerate coordinates:
\[
  \Var f(g)
  =\int_0^\infty e^{-t}\,
    \E\bigl[\nabla f(g_t^{(1)})\cdot\nabla f(g_t^{(2)})\bigr]\,\dd t.
\]
The chain rule rewrites the inner product in terms of derivatives with respect to the field coordinates as
\begin{equation}\label{eq:HS}
\sum_{\sigma,\sigma'\in\Sigma_N}
\Cov\bigl(H(\sigma),H(\sigma')\bigr)
\,\partial_{H(\sigma)}F\bigl(H_t^{(1)}\bigr)
\,\partial_{H(\sigma')}F\bigl(H_t^{(2)}\bigr).
\end{equation}
This coordinate derivation also justifies the formula when the vector
$(H(\sigma))_{\sigma\in\Sigma_N}$ is degenerate.

Differentiating
$F(H)=N^{-1}\log\sum_\tau e^{\beta H(\tau)}$ gives
\[
  \partial_{H(\sigma)}F(H)
  =\frac{\beta}{N}\,
    \frac{e^{\beta H(\sigma)}}{\sum_\tau e^{\beta H(\tau)}}.
\]
Evaluating the two derivatives in \eqref{eq:HS} under the two coupled Gibbs
measures and using
$\Cov(H(\sigma),H(\sigma'))=N\xi(R_{\sigma,\sigma'})$ turns
\eqref{eq:HS} into
\[
  \frac{\beta^2}{N}\,\big\langle\xi(R_{1,2})\big\rangle_t.
\]
Integration and disorder expectation give \eqref{eq:var-id-general}. For SK,
$\xi(r)=r^2/2$, which gives \eqref{eq:var-id-SK}.
\end{proof}

% ---------- SK at the critical inverse temperature ----------
\section{SK at criticality}\label{sec:SK-critical}
We work at $\beta_c=1$, $h=0$, and use the OU notation from Section~\ref{sec:setup}: $\rho_t=e^{-t}$ and
\[
\chi_N(t):=N\,\E\langle R_{1,2}^2\rangle_t.
\]
We also write
\[
  \chi_{\mathrm{SG}}(N):=N\,\E\langle R_{1,3}^2\rangle
\]
for the usual \emph{single-system} spin-glass susceptibility (two independent replicas of one Gibbs measure at $(\beta_c,h=0)$). Note that $\chi_N(0)=\chi_{\mathrm{SG}}(N)$.

\paragraph{Cross-system fluctuation $Q_e$.}
The next two lemmas introduce the observable used in the incident-edge identity. Its role here is to express the coupled susceptibility. We use four replicas, with $(\sigma^{(1)},\sigma^{(3)})$ sampled from system~$1$ and $(\sigma^{(2)},\sigma^{(4)})$ sampled from system~$2$. The cross-system fluctuation
\begin{equation}\label{eq:Qe-def}
  Q_e\;:=\;\frac{1}{\sqrt2}\bigl(R_{1,2}-R_{1,4}\bigr)
\end{equation}
subtracts two overlaps sharing the same first-system replica. At zero field, the mixed term in the square vanishes by replica independence and spin-flip symmetry, giving the OU susceptibility.

\begin{lemma}[Identity for $\chi_N(t)$ and local normalization]\label{lem:cross-only-vertex}
At $h=0$ one has
\begin{equation}\label{eq:Qe-chi}
  \E\langle Q_e^2\rangle_t\;=\;\E\langle R_{1,2}^2\rangle_t,
  \qquad\text{hence}\qquad
  \chi_N(t)=N\,\E\langle Q_e^2\rangle_t.
\end{equation}
Moreover, for every site $i$,
\begin{equation}\label{eq:vertex-one}
  \E\Big\langle \frac12\bigl(\sigma_i^{(1)}\sigma_i^{(2)}-\sigma_i^{(1)}\sigma_i^{(4)}\bigr)^2\Big\rangle_t\;=\;1.
\end{equation}
\end{lemma}

\begin{proof}
Expand
\[
  \langle Q_e^2\rangle_t
  =\frac12\langle R_{1,2}^2\rangle_t+\frac12\langle R_{1,4}^2\rangle_t-\langle R_{1,2}R_{1,4}\rangle_t.
\]
By symmetry between replicas $2$ and $4$, $\E\langle R_{1,4}^2\rangle_t=\E\langle R_{1,2}^2\rangle_t$.
For the cross term, condition on $(H_t^{(1)},H_t^{(2)})$.
Then replicas $2$ and $4$ are independent samples from the same Gibbs measure of system~$2$, so
\[
  \langle R_{1,2}R_{1,4}\rangle_t
  =\frac{1}{N^2}\sum_{i,j}\langle \sigma_i^{(1)}\sigma_j^{(1)}\rangle_t\,
     \langle \sigma_i^{(2)}\rangle_t\,\langle \sigma_j^{(4)}\rangle_t.
\]
At $h=0$, $\langle \sigma_i^{(2)}\rangle_t=\langle \sigma_j^{(4)}\rangle_t=0$ by the global flip symmetry, hence the cross term vanishes and \eqref{eq:Qe-chi} follows.

For \eqref{eq:vertex-one}, note that
\[
  \frac12(\sigma_i^{(1)}\sigma_i^{(2)}-\sigma_i^{(1)}\sigma_i^{(4)})^2
  =\frac12(\sigma_i^{(2)}-\sigma_i^{(4)})^2
  =1-\sigma_i^{(2)}\sigma_i^{(4)}.
\]
Conditioning on the disorder and using independence of replicas $2$ and $4$,
$\langle \sigma_i^{(2)}\sigma_i^{(4)}\rangle_t=\langle \sigma_i\rangle_t^2=0$ at $h=0$.
\end{proof}

\begin{lemma}[Local mixed derivatives at zero field]\label{lem:Qe-local-vertex}
Let
\[
q_j^e=\frac{1}{\sqrt2}\Big(\sigma_j^{(1)}\sigma_j^{(2)}-\sigma_j^{(1)}\sigma_j^{(4)}\Big).
\]
Introduce local fields $x=(x_1,x_2,x_3,x_4)$ at site $j$ by adding
$\sum_{a=1}^4 x_a\sigma_j^{(a)}$ to the four-replica Hamiltonian, and write
$\langle\cdot\rangle_{t,x}$ for the corresponding Gibbs expectation. Then
\begin{equation}\label{eq:Qe-local-Hessian}
\left.
\frac{\partial^2}{\partial x_1\partial x_2}
\langle q_j^e\rangle_{t,x}
\right|_{x=0}=\frac{1}{\sqrt2},
\qquad
\left.
\frac{\partial^2}{\partial x_1\partial x_4}
\langle q_j^e\rangle_{t,x}
\right|_{x=0}=-\frac{1}{\sqrt2},
\end{equation}
and, apart from the symmetric counterparts of these mixed derivatives, all other second derivatives vanish.
Consequently the two local coefficients have squares summing to
$(1/\sqrt2)^2+(-1/\sqrt2)^2=1$.
\end{lemma}

\begin{proof}
For each fixed disorder realization the replicas are independent Gibbs samples
(possibly from the same disorder copy). At zero external field,
$\langle\sigma_j^{(a)}\rangle_t=0$ for every replica $a$ by global spin-flip
symmetry. Moreover, differentiating with respect to the local field of the
same replica gives
\[
\left.\partial_{x_a}\langle\sigma_j^{(a)}\rangle_{t,x}\right|_{x=0}
=\langle(\sigma_j^{(a)})^2\rangle_t-
 \langle\sigma_j^{(a)}\rangle_t^2=1.
\]
Since
\[
\langle \sigma_j^{(1)}\sigma_j^{(2)}\rangle_{t,x}
=\langle\sigma_j^{(1)}\rangle_{t,x_1}
 \langle\sigma_j^{(2)}\rangle_{t,x_2},
\]
and similarly for $(1,4)$, we get
\[
\left.\partial_{x_1}\partial_{x_2}
\langle\sigma_j^{(1)}\sigma_j^{(2)}\rangle_{t,x}\right|_{x=0}=1,
\qquad
\left.\partial_{x_1}\partial_{x_4}
\langle\sigma_j^{(1)}\sigma_j^{(4)}\rangle_{t,x}\right|_{x=0}=1.
\]
Inserting the signs in the definition of $q_j^e$ gives \eqref{eq:Qe-local-Hessian}.
Apart from these mixed derivatives and their symmetric counterparts, every second derivative either differentiates a factor absent from the monomial or leaves an unpaired expectation of one spin at zero field, and hence vanishes.
\end{proof}

The estimate in the critical window starts from the exact direct $Q_e$
identity, which is Proposition~\ref{prop:Qe-direct-ladder} in Appendix~\ref{app:replicon}. This is a scalar identity for the susceptibility $\chi_N(t)$ after all ordered incident-edge pairs have been summed:
\[
  \theta_N(t)\,\chi_N(t)
  =1+\frac{\rho_t}{N}
  +\mathfrak R_N^{\mathrm{Schur}}(t)+\mathfrak R_N^{\mathrm{interp}}(t),
  \qquad
  \theta_N(t):=1-\rho_t+\frac{2\rho_t}{N}.
\]
The term $1$ is the summed diagonal contribution $C_{jj}^e=1$ from
Lemma~\ref{lem:cross-only-vertex}.  The squares of the two independent mixed derivatives in Lemma~\ref{lem:Qe-local-vertex} sum to one. This calculation is not used in deriving the identity.  The term
$\rho_t/N$ and the denominator $\theta_N(t)$ are exact finite-$N$
corrections from excluding the fixed incident site in the edge summation.
See the summation over indices distinct from the fixed site in the proof of
Proposition~\ref{prop:Qe-direct-ladder}.  The terms
$\mathfrak R_N^{\mathrm{Schur}}$ and $\mathfrak R_N^{\mathrm{interp}}$
are the global sums in~\eqref{eq:global-remainders}. Their per-pair
definitions are \eqref{eq:R-schur} and~\eqref{eq:R-interp}.  The first
comes from the Schur-complement part of the incident-edge identity, and
the second from replacing full OU covariance by partial incident-edge
covariance during the interpolation.
Both remainders are nonpositive,
$\mathfrak R_N^{\mathrm{Schur}}(t),\,\mathfrak R_N^{\mathrm{interp}}(t)\le 0$
(Lemma~\ref{lem:remainder-signs}), so unconditionally
\begin{equation}\label{eq:upper-profile-main}
  \chi_N(t)\;\le\;\frac{1+\rho_t/N}{\theta_N(t)}
  \qquad\text{(Corollary~\ref{cor:upper-profile}).}
\end{equation}
Assumption~\ref{ass:SK-remainder-closure}, verified directly in
Proposition~\ref{prop:cubic-from-moments}, is the condition needed to
integrate this identity with bounded error through the critical window.
It is weaker than a pointwise rational approximation but has the same consequence for the OU integral:
above the cutoff scale $t_\ast\asymp \chi_{\mathrm{SG}}(N)^{-1}$ the explicit leading term behaves like $(1-e^{-t})^{-1}$, while the OU monotonicity bound
$\chi_N(t)\le \chi_{\mathrm{SG}}(N)$ (Lemma~\ref{lem:chi-CS}) controls the interval below that cutoff.
This is the reason for the logarithm $\log \chi_{\mathrm{SG}}(N)$
in Theorem~\ref{thm:intro-SK} (see also Remark~\ref{rem:boundary-layer}).

For comparison, the following remark records what a stronger pointwise signed-remainder estimate would imply.
\begin{remark}[A pointwise remainder bound would give a rational approximation]\label{cor:scalar-ladder}
The exact identity~\eqref{eq:Qe-exact-identity} always holds. If, in addition, one could prove
\[
  \left|\mathfrak R_N^{\mathrm{Schur}}(t)+\mathfrak R_N^{\mathrm{interp}}(t)\right|
  \le C(c)/N
  \qquad\text{for }t\in[c/\chi_{\rm SG}(N),1],
\]
which is the pointwise estimate stated in~\eqref{eq:Qe-direct-ladder}, then
\begin{equation}\label{eq:scalar-ladder-chi}
  \Bigl(1-\rho_t+\frac{2\rho_t}{N}\Bigr)\chi_N(t)=1+O(N^{-1})
\end{equation}
uniformly on that window, equivalently
\begin{equation}\label{eq:scalar-resolvent}
  \chi_N(t)=\frac{1+O(N^{-1})}{1-\rho_t+2\rho_t/N}.
\end{equation}
At criticality, however, this pointwise estimate cannot hold uniformly
on the stated window for any fixed $c>1$. Write
$\chi=\chi_{\rm SG}(N)$. The established critical scaling gives
$\chi\to\infty$ and $\chi/N\to0$. If the estimate held, the exact
identity would imply
\[
  \frac{\chi_N(jc/\chi)}{\chi}\longrightarrow\frac1{jc}
  \qquad\text{for every fixed integer }j\ge1.
\]
The Hermite expansion in Lemma~\ref{lem:chi-CS} has the form
$\chi_N(t)=\sum_{k\ge0}b_{N,k}e^{-kt}$, with $b_{N,k}\ge0$ and
$\sum_kb_{N,k}=\chi$. Define a random integer $K_N$ by
$\mathbb P(K_N=k)=b_{N,k}/\chi$, and put
$X_N=e^{-cK_N/\chi}\in[0,1]$. Then
$\mathbb E X_N^j\to1/(jc)$. For every fixed integer $m\ge1$,
the binomial expansion therefore gives
\[
  \mathbb E\bigl[1-(1-X_N)^m\bigr]
  \longrightarrow\frac1c\sum_{j=1}^m\frac1j.
\]
The left side is at most $1$, whereas the sum on the right diverges
as $m\to\infty$. This is a contradiction. The main results use the
integrated remainder condition rather than this pointwise estimate.
\end{remark}

\begin{remark}[Crossover regime and the $t=0$ endpoint]%
\label{rem:endpoint-t0}
The pointwise estimate in Remark~\ref{cor:scalar-ladder} cannot hold down to $t=0$. The preceding remark shows that it also cannot hold uniformly on the stated critical window after any fixed cutoff $c/\chi_{\rm SG}(N)$. At $t=0$ one has $\rho_0=1$ and $\chi_N(0)=\chi_{\mathrm{SG}}(N)$. Hence an identity of the form
$(1-\rho_t+2\rho_t/N)\chi_N(t)=1+\mathrm{Err}_t$ would require
\[
  \mathrm{Err}_0=2\chi_{\mathrm{SG}}(N)/N-1.
\]
This error is not $O(N^{-1})$ when $\chi_{\mathrm{SG}}(N)\gg1$. The variance proof therefore does not use a pointwise rational approximation at the endpoint. Instead it splits the OU integral at $t_\ast=c/\chi_{\mathrm{SG}}(N)$ and controls $(0,t_\ast]$ by the OU monotonicity bound $\chi_N(t)\le\chi_{\mathrm{SG}}(N)$ from Lemma~\ref{lem:chi-CS}. Above that scale, the explicit leading term in~\eqref{eq:Qe-exact-identity} has the size $(1-e^{-t})^{-1}$ up to finite-$N$ corrections.
\end{remark}

\begin{remark}[Formal interpretation of the geometric series]\label{rem:geometric-series}
This remark is only interpretive and is not used in any proof.  If the
signed remainders in~\eqref{eq:Qe-exact-identity} are ignored, then
\[
  \chi_N(t)\approx
  \left(1+\frac{\rho_t}{N}\right)
  \sum_{n=0}^\infty[\rho_t(1-2/N)]^n.
\]
For each fixed $N\ge2$, the series does not diverge as
$\rho_t\uparrow1$: its sum tends to $N/2$ (and the displayed expression
to $(N+1)/2$).  Its growth is instead a joint finite-size effect.  As
$N\to\infty$ and $\rho_t\uparrow1$, the series is of order
$(1-\rho_t+2\rho_t/N)^{-1}$. On the critical cutoff scale
$1-\rho_t\asymp N^{-1/3}$, it is of order $N^{1/3}$.  The $n$-th term
represents $n$ iterations of the leading shared-disorder response, with
the factor $1-2/N$ coming from the same incident-edge exclusion that
produces $\theta_N(t)$.  A rigorous lower-bound counterpart of the first
terms of this picture appears in~\cite{Schertzer2026CriticalSK}: in
the Hermite expansion of the variance~\cite{Chatterjee2009}, the
simple disorder cycles of length $k$ jointly contribute
$\tfrac1{2k}+o(1/k)$ to $\Var(\log Z_N)$ for $k\le\log\log N$.
\end{remark}

\begin{lemma}[OU monotonicity of the coupled susceptibility]\label{lem:chi-CS}
Let $\widehat\chi_N(\rho)$ denote the coupled susceptibility when the two
off-diagonal disorder vectors have correlation $\rho\in[0,1]$.  Then
$\rho\mapsto\widehat\chi_N(\rho)$ is nondecreasing.  Consequently,
\[
  \chi_N(t)=\widehat\chi_N(e^{-t})
  \le \widehat\chi_N(1)=\chi_{\rm SG}(N)
  \qquad(t\ge0).
\]
\end{lemma}

\begin{proof}
For $G_{ij}(J)=\langle\sigma_i\sigma_j\rangle_J$, let
$G_{ij}=\sum_\alpha a_{ij,\alpha}H_\alpha$ be its orthonormal Hermite
expansion in the independent off-diagonal Gaussian coordinates.  The
Hermite--Mehler formula gives
\[
  \widehat\chi_N(\rho)
  =\frac1N\sum_{i,j}\sum_\alpha a_{ij,\alpha}^2\rho^{|\alpha|}.
\]
All coefficients are nonnegative, so the function is nondecreasing in
$\rho$.  This is the corresponding specialization of the expansion used in
Lemma~\ref{lem:partial-hermite}.  The monotonicity is proved in~\cite[Theorem~3.1]{Chatterjee2009} for all
even overlap moments. The Hermite--Mehler argument above keeps the treatment
self-contained.
\end{proof}

\begin{proposition}[Exact integral gap at criticality]
\label{prop:exact-critical-gap}
Write $\chi=\chi_{\rm SG}(N)$ and $r_*=1-\chi^{-1}$. Define
\begin{equation}\label{eq:exact-critical-gap}
\begin{split}
  \mathcal G_N:={}&
  \int_0^{r_*}\left(\frac1{1-r}-\widehat\chi_N(r)\right)\,\dd r
  +\int_{r_*}^1\left(\chi-\widehat\chi_N(r)\right)\,\dd r.
\end{split}
\end{equation}
Both integrands are nonnegative, and
\begin{equation}\label{eq:critical-gap-identity}
  2N^2\Var F_N(1)=1+\log\chi-\mathcal G_N,
  \qquad
  2N^2\Var F_N^{\rm off}(1)=\log\chi-\mathcal G_N.
\end{equation}
\end{proposition}

\begin{proof}
For $N\ge2$ and $0\le r<1$, Corollary~\ref{cor:upper-profile} gives
\[
  \widehat\chi_N(r)
  \le\frac{1+r/N}{1-r+2r/N}\le\frac1{1-r}.
\]
The last inequality follows on multiplying by the positive denominators.
Lemma~\ref{lem:chi-CS} also gives $\widehat\chi_N(r)\le\chi$.
The two bounds agree at $r_*$. Hence~\eqref{eq:exact-critical-gap}
is nonnegative. Since
\[
  \int_0^{r_*}\frac{\dd r}{1-r}+\chi(1-r_*)=\log\chi+1,
\]
the OU variance identity proves the first equality
in~\eqref{eq:critical-gap-identity}. Lemma~\ref{lem:GOE-vs-offdiag}
proves the second. For $N=1$, $\widehat\chi_1(r)=\chi=1$,
$\mathcal G_1=0$, and both equalities hold directly.
\end{proof}

\begin{proposition}[Limit of the integral gap]
\label{prop:limiting-critical-gap}
Let $C_\chi$ be the susceptibility constant in~\eqref{eq:chi-limit}.
Then the exact integral \eqref{eq:exact-critical-gap} satisfies
\begin{equation}\label{eq:limiting-critical-gap}
  \mathcal G_N
  =\log C_\chi
   -2\Bigl(N^2\Var F_N^{\rm off}(1)-\tfrac16\log N\Bigr)+o(1),
\end{equation}
and $\mathcal G_N$ is bounded. Consequently $\mathcal G_N$ converges
if and only if the finite part of the variance converges, that is, if
and only if a constant $c_{\rm Ising}$ satisfying
\eqref{eq:intro-variance-constant} exists. In that case
\[
  \lim_{N\to\infty}\mathcal G_N
  =\log C_\chi-2c_{\rm Ising}\ge0,
\]
and in particular $c_{\rm Ising}\le\tfrac12\log C_\chi$.
\end{proposition}

\begin{proof}
Since $C_\chi>0$, equation~\eqref{eq:chi-limit} implies
$\log\chi_{\rm SG}(N)=\tfrac13\log N+\log C_\chi+o(1)$.
Substituting this into \eqref{eq:critical-gap-identity} gives
\eqref{eq:limiting-critical-gap}. For any fixed $c>1$,
\eqref{eq:critical-gap-identity} and
Proposition~\ref{prop:saturation-remainder} give
$\mathcal G_N=\mathcal D_N^{(c)}+O_c(1)$, and
Proposition~\ref{prop:cubic-from-moments} gives
$\mathcal D_N^{(c)}=O_c(1)$. Hence $\mathcal G_N$ is bounded, the
bracket in \eqref{eq:limiting-critical-gap} is bounded, and the two
convergence statements are equivalent, with the stated limit. Its sign
follows from $\mathcal G_N\ge0$ for every $N$.
\end{proof}

\begin{remark}[Relation between the two integral remainders]
\label{rem:gap-comparison}
The quantity $\mathcal G_N$ is defined without an arbitrary cutoff.
Combining~\eqref{eq:critical-gap-identity} with
Proposition~\ref{prop:saturation-remainder} below gives
$\mathcal G_N-\mathcal D_N^{(c)}=O_c(1)$ for each fixed $c>1$.
They therefore have equivalent boundedness properties. The remainder
$\mathcal D_N^{(c)}$ retains the separate Schur and interpolation
contributions of the local identity, whereas $\mathcal G_N$ records
the exact gap in the susceptibility bound.
\end{remark}

\begin{remark}[Heuristic cutoff interpretation]\label{rem:boundary-layer}
Lemma~\ref{lem:chi-CS} says that decorrelating the disorder decreases the
coupled susceptibility.  Therefore a profile $\chi_N(t)\asymp t^{-1}$ cannot
persist below $t\asymp\chi_{\rm SG}(N)^{-1}$, and the OU integral naturally
produces $\log\chi_{\rm SG}(N)$ rather than $\log N$ directly.
Lemma~\ref{lem:half-log} records the corresponding model computation: the
scalar profile $e^{-t}/(1-e^{-t})$ integrates to $\log M+O(1)$ above the
cutoff $1/M$. Lemma~\ref{lem:explicit-window-integral} is its finite-$N$
version with the exact denominator $\theta_N(t)$.
\end{remark}

\begin{lemma}[Logarithmic integral]\label{lem:half-log}
For every $M\ge1$,
\[
  \int_{1/M}^{1}\frac{e^{-t}}{1-e^{-t}}\,\dd t=\log M+O(1).
\]
\end{lemma}

\begin{proof}
Since $\frac{\dd}{\dd t}\log(1-e^{-t})=e^{-t}/(1-e^{-t})$,
\[
  \int_{1/M}^{1}\frac{e^{-t}}{1-e^{-t}}\,\dd t
  =\log(1-e^{-1})-\log(1-e^{-1/M}).
\]
For $x\in(0,1]$, $x/2\le1-e^{-x}\le x$, which gives the claim.
\end{proof}

\begin{lemma}[Explicit critical-window integral]\label{lem:explicit-window-integral}
Fix $c>1$ and write $\chi=\chi_{\rm SG}(N)$.  With the preceding convention for an empty integration interval,
\begin{equation}\label{eq:explicit-window-integral}
  \int_{c/\chi}^{1}e^{-t}
  \frac{1+\rho_t/N}{\theta_N(t)}\,\dd t
  =\log\chi+O_c(1).
\end{equation}
\end{lemma}

\begin{proof}
The cases $N\le2$ are immediate: if $N=1$ the integral is empty, while for $N=2$ the denominator equals $1$ and both sides are $O_c(1)$.  Assume $N\ge3$.  When $c/\chi\ge1$, both sides are $O_c(1)$.  Otherwise put $a=c/\chi$ and
substitute $\rho=e^{-t}$:
\[
  \int_a^1e^{-t}\frac{1+\rho_t/N}{\theta_N(t)}\,\dd t
  =\int_{e^{-1}}^{e^{-a}}\frac{1+\rho/N}{1-\rho+2\rho/N}\,\dd\rho.
\]
The numerator correction contributes $O_c(1)$.  For the main term,
\[
  \int_{e^{-1}}^{e^{-a}}\frac{\dd\rho}{1-\rho(1-2/N)}
  =\frac{N}{N-2}
   \left[\log\frac1{1-\rho(1-2/N)}\right]_{e^{-1}}^{e^{-a}}.
\]
Since $1\le\chi\le N$ (the diagonal of $\E\langle R_{1,3}^2\rangle$ gives
$\chi\ge1$, and $R_{1,3}^2\le1$ gives $\chi\le N$) and
$1-e^{-a}(1-2/N)=c/\chi+2/N+O_c(\chi^{-2})$, the upper endpoint is
$\log\chi+O_c(1)$. The lower endpoint is $O(1)$.
\end{proof}

\begin{lemma}[Cutoff stability of the integrated remainder]
\label{lem:cutoff-stability}
For fixed $1<c_1<c_2$,
\[
  \bigl|\mathcal D_N^{(c_1)}-\mathcal D_N^{(c_2)}\bigr|
  =O_{c_1,c_2}(1).
\]
Thus the bounded-remainder assumption for one fixed cutoff is equivalent to the
same assumption for every fixed cutoff.
\end{lemma}

\begin{proof}
If $c_1/\chi_{\rm SG}(N)\ge1$, both remainders vanish.  Otherwise their
difference is the remainder integral over
\[
  [c_1/\chi_{\rm SG}(N),\,
    \min\{c_2/\chi_{\rm SG}(N),1\}].
\]
By the exact identity and $\chi_N(t)\ge0$,
\[
  0\le
  \frac{-\mathfrak R_N^{\rm Schur}(t)-\mathfrak R_N^{\rm interp}(t)}
       {\theta_N(t)}
  =\frac{1+\rho_t/N}{\theta_N(t)}-\chi_N(t)
  \le\frac{1+\rho_t/N}{\theta_N(t)}.
\]
The explicit rational bound integrates to $O_{c_1,c_2}(1)$ on this interval, by the same endpoint computation as
Lemma~\ref{lem:explicit-window-integral}.
\end{proof}

\begin{proposition}[Integrated signed-remainder decomposition and saturation equivalence]
\label{prop:saturation-remainder}
For every fixed $c>1$,
\begin{equation}\label{eq:saturation-remainder-identity}
  2N^2\Var F_N(1)
  =\log\chi_{\rm SG}(N)-\mathcal D_N^{(c)}+O_c(1).
\end{equation}
Consequently,
\begin{equation}\label{eq:saturation-equivalence}
  \mathcal D_N^{(c)}=O_c(1)
  \quad\Longleftrightarrow\quad
  \Var F_N(1)=\frac{\log\chi_{\rm SG}(N)}{2N^2}+O(N^{-2}).
\end{equation}
More generally, if
$\mathcal D_N^{(c)}=\lambda_N\log\chi_{\rm SG}(N)+O_c(1)$, then
\begin{equation}\label{eq:partial-depletion-coefficient}
  \Var F_N(1)
  =\frac{1-\lambda_N}{2N^2}\log\chi_{\rm SG}(N)+O_c(N^{-2}).
\end{equation}
\end{proposition}

\begin{proof}
Write $\chi=\chi_{\rm SG}(N)$ and $a=c/\chi$.  If $a\ge1$, then
$\log\chi=O_c(1)$ and Lemma~\ref{lem:chi-CS} gives
$2N^2\Var F_N(1)\le\chi=O_c(1)$, so the identity is absorbed into the
error term.  If $a<1$, the exact incident-edge identity gives
\[
  \int_a^1e^{-t}\chi_N(t)\,\dd t
  =\int_a^1e^{-t}\frac{1+\rho_t/N}{\theta_N(t)}\,\dd t
   -\mathcal D_N^{(c)}.
\]
The first term is $\log\chi+O_c(1)$ by
Lemma~\ref{lem:explicit-window-integral}.  The integral over $[0,a]$ is
at most $a\chi=c$ by Lemma~\ref{lem:chi-CS}, while the tail $t\ge1$ is
$O(1)$ by Corollary~\ref{cor:upper-profile}.  The OU variance identity now
gives \eqref{eq:saturation-remainder-identity}. The remaining assertions are
immediate.
\end{proof}

\begin{theorem}[Finite-size variance bound]
\label{thm:critical-variance-upper}
At $\beta=1$, $h=0$, and every $N\ge1$,
\[
  \Var F_N(1)\le\frac{1+\log\chi_{\rm SG}(N)}{2N^2},
  \qquad
  \Var F_N^{\rm off}(1)\le\frac{\log\chi_{\rm SG}(N)}{2N^2}.
\]
The susceptibility consequences stated in
Theorem~\ref{thm:intro-SK-upper} follow.
\end{theorem}

\begin{proof}
Drop $\mathcal G_N\ge0$ in~\eqref{eq:critical-gap-identity}.
The limiting bound~\eqref{eq:intro-constant-upper} follows by inserting
$\log\chi_{\rm SG}(N)=\tfrac13\log N+\log C_\chi+o(1)$
from~\eqref{eq:chi-limit}.
\end{proof}

\begin{theorem}[Saturation criterion]
\label{thm:critical-variance-constant}
Under Assumption~\ref{ass:SK-remainder-closure},
\[
  \Var F_N(1)=\frac{\log\chi_{\rm SG}(N)}{2N^2}+O(N^{-2}).
\]
\end{theorem}

\begin{proof}
Combine Proposition~\ref{prop:saturation-remainder} with
Assumption~\ref{ass:SK-remainder-closure}. Cutoff stability makes the choice
of fixed $c>1$ immaterial.
\end{proof}

\begin{remark}[Meaning of saturation]\label{rem:saturation}
The sign lemma shows that the remainder decreases the explicit
logarithmic contribution. Saturation means this decrease is bounded.
Proposition~\ref{prop:cubic-from-moments} verifies the required bound
directly using the cited moment and cavity estimates. Conversely, any independent
variance lower bound supplies an upper bound on the integrated remainder
through Proposition~\ref{prop:saturation-remainder}.
\end{remark}

Before stating a sufficient condition, recall the mixed cubic functional
$\mathfrak C_N(t)$ from \eqref{eq:mixed-cubic-functional}.  For a fixed site
$j$, the expectation $\E_{j,\lambda}$ changes only the cross-covariance of the
edges incident to $j$. All one-system marginals remain the ordinary SK law.

\begin{proposition}[Integrated cubic criterion for the Schur remainder]
\label{prop:mixed-cubic-schur}
Let $\chi=\chi_{\rm SG}(N)$ and fix $c>1$. If
\begin{equation}\label{eq:integrated-cubic-control}
  \int_{c/\chi}^{1}e^{-t}
  \frac{\rho_t N\mathfrak C_N(t)}{\theta_N(t)}\,\dd t
  =O_c(1),
\end{equation}
then
\[
  \int_{c/\chi}^{1}e^{-t}
  \frac{-\mathfrak R_N^{\rm Schur}(t)}{\theta_N(t)}\,\dd t
  =O_c(1).
\]
Consequently the full signed-remainder assumption holds: the interpolation
part is bounded by Lemma~\ref{lem:interp-integrated-app} and the classical
estimate $\chi=O(N^{1/2})$~\cite[Proposition~11.7.8]{Talagrand2011Vol2}.
If in addition $\chi\le K N^{1/3}$, a sufficient pointwise condition for
\eqref{eq:integrated-cubic-control} is
\begin{equation}\label{eq:mixed-cubic-control}
  \frac1N\sum_{j=1}^N\int_0^1
  \E_{j,s\rho_t}
  \left|\left\langle R_{12}R_{13}q_j^{23}\right\rangle\right|\,\dd s
  \le C_*\frac{1+\chi_N(t)}{N}
       \left(\frac{\chi}{N}\right)^{1/2}
\end{equation}
uniformly for $t\in[c/\chi,1]$.
\end{proposition}

\begin{proof}
Proposition~\ref{prop:schur-cubic-control} gives
\[
  -\mathfrak R_N^{\rm Schur}(t)
  \le2\rho_tN\mathfrak C_N(t).
\]
Multiplication by $e^{-t}/\theta_N(t)$ and integration prove the
claim directly from~\eqref{eq:integrated-cubic-control}. No endpoint
error or susceptibility hypothesis is needed for this implication.

To see that \eqref{eq:mixed-cubic-control} suffices, insert it into the left
side of \eqref{eq:integrated-cubic-control} and use
$\chi_N(t)\lesssim\theta_N(t)^{-1}$ and
$\theta_N(t)\asymp t+N^{-1}$ on $[c/\chi,1]$.  The result is bounded by
\[
  \left(\frac{\chi}{N}\right)^{1/2}
  \left(\log\chi+\chi\right)=O_{c,K,C_*}(1).
\]
\end{proof}

\begin{remark}[Scope of the pointwise criterion]\label{rem:FM4-route}
The factor $(\chi/N)^{1/2}$ in \eqref{eq:mixed-cubic-control} has the
scale of an ordinary overlap. This observation does not prove that the
fixed-site factor has that gain under the partial interpolation.
The integrated criterion is weaker and is established in
Proposition~\ref{prop:cubic-from-moments}. The pointwise estimate is
neither proved nor needed for that conclusion.
\end{remark}

\section{Conclusions and outlook}\label{sec:conclusions}

The variance bound can be stated without an asymptotic error or an
arbitrary cutoff. In the off-diagonal convention,
\[
  2N^2\Var F_N^{\rm off}(\beta)
  \le\log\chi_{\rm SG}(N;\beta)
    +(\beta^2-1)\bigl(\chi_{\rm SG}(N;\beta)-1\bigr).
\]
At criticality, Proposition~\ref{prop:exact-critical-gap} strengthens
this inequality to the exact relation
\[
  2N^2\Var F_N^{\rm off}(1)
    =\log\chi_{\rm SG}(N)-\mathcal G_N,
  \qquad \mathcal G_N\ge0.
\]
The incident-edge identity gives a second decomposition, in which
the Schur and interpolation remainders retain their separate meanings.
The cubic estimate of Proposition~\ref{prop:schur-cubic-control}
has no endpoint error: positivity of the correlation matrices makes
the endpoint correction favorable after disorder averaging.

The leading critical variance and susceptibility scales are established
\cite{DuHuang2026,YimingChen2026,DuHuang2026Overlap}.
The reverse Schur comparison in Proposition~\ref{prop:reverse-schur}
and the critical inputs in Proposition~\ref{prop:cubic-from-moments}
give an integrated cubic bound without invoking the final variance
theorem. This quantifies the separate local remainder structure, while
the total remainder bound proves saturation directly. A verification
within the signed-remainder framework itself, avoiding the cavity
expansion, remains open. The finite-size upper inequality itself uses
none of those asymptotic inputs.

Proposition~\ref{prop:limiting-critical-gap} reduces the limit of the
exact integral to the existence of a constant term in the critical
variance. Whenever the limit~\eqref{eq:intro-variance-constant} exists,
\begin{equation}\label{eq:constant-gap-target}
  \mathcal G_N\longrightarrow
  \log C_\chi-2c_{\rm Ising}\ge0.
\end{equation}
The two constants then measure the limiting difference between
the susceptibility bound and the actual variance. The existence and
the value of $c_{\rm Ising}$ are the subject of a companion paper in
preparation~\cite{Tierz2026Constants}. No result of the present paper
depends on them.

A further question is how this total difference is distributed over
disorder perturbations at the critical scale. A sufficient form of
profile convergence can be stated precisely. Put
$L_N=N^{1/3}$ and
\[
  \psi_N(s)=N^{-1/3}\widehat\chi_N(1-sN^{-1/3}),
  \qquad 0\le s\le L_N.
\]
The variable in $\widehat\chi_N$ is the disorder correlation, not the
OU time. The variance identity gives exactly
\[
  2N^2\Var F_N(1)=\int_0^{L_N}\psi_N(s)\,\dd s.
\]
If the functions
\[
  h_N(s)=\mathbf 1_{[0,L_N]}(s)
       \left(\psi_N(s)-\frac1{1+s}\right)
\]
converge in $L^1([0,\infty))$ to some $h$, then
\[
  N^2\Var F_N^{\rm off}(1)-\frac16\log N
  \longrightarrow\frac12\int_0^\infty h(s)\,\dd s-\frac12.
\]
Indeed, the subtracted profile integrates to
$\log(1+L_N)=\tfrac13\log N+o(1)$. The final $-1/2$ removes the
GOE diagonal variance. The overlap limit in~\cite{DuHuang2026Overlap}
identifies $\psi_N(0)\to C_\chi$, but does not by itself establish
this profile or its integrated convergence. If such convergence holds,
the variance constant exists and the limit satisfies
$\int_0^\infty h(s)\,\dd s=1+2c_{\rm Ising}$.
Conversely, existence of the variance constant is equivalent to
convergence of the integrals $\int_0^{L_N}h_N(s)\,\dd s$, whose limit
is then $1+2c_{\rm Ising}$, and does not by itself imply convergence
of the functions $h_N$.

Appendix~\ref{app:temperature-extensions} establishes the variance upper
bound throughout every bounded $N^{-1/3}$ temperature window by extending
the entropy argument of~\cite{YimingChen2026}. The ordinary overlap
moments remain stable on the wider interval $N^{-4/9}$ obtained by
combining the critical linear exponential tail with a quartic estimate
for larger overlaps. This improves the $N^{-2/3}$ interval obtained from
the linear tail alone. It still falls short of the full $N^{-1/3}$
overlap window. A uniform critical exponential moment of
$N^{2/3}R_{12}^2$ would suffice to reach a small interval at that scale.
The larger free-energy fluctuation window discussed
in~\cite[Section~1.4]{DuHuang2026Overlap} is a distinct target. Extensions to nonzero field require replacing the parity
cancellation used here, and should be compared with the distinct
fluctuation results of~\cite{ChenDeyPanchenko2017ExternalField}.
For multi-species models~\cite{BarraContucciMingioneTantari2015,Panchenko2015MSK,CollinsWoodfinLe2025MSK},
one would first need a matrix version of the incident-edge identity
and its remainder sign. We leave these extensions separate from the
finite-size inequalities proved here.

\appendix

\section{Incident-edge and Ornstein--Uhlenbeck identities}\label{app:replicon}

This appendix contains the technical SK inputs used in the main text.  The OU variance representation is the Helffer--Sj\"ostrand covariance formula on Gaussian space~\cite{HelfferSjostrand1994,BakryGentilLedoux2014,Chatterjee2014}. The incident-edge identity is a finite-dimensional application of Price's Gaussian covariance-differentiation formula~\cite{Price1958,Chatterjee2014}. The interpolation monotonicity uses the Hermite--Mehler expansion of correlated Gaussian functions~\cite{Janson1997,Chatterjee2014}, and the Schur estimates use Schur complements and the Hadamard-product theorem~\cite{Schur1911,HornJohnson2013}.

\subsection{Fields coupled by Ornstein--Uhlenbeck interpolation and the variance identity}

The OU variance representation has already been proved in
Theorem~\ref{thm:var-identity}, by applying the Gaussian covariance formula
in underlying independent coordinates and using the chain rule to express it in the field coordinates.  We record only the two formulas used below.  For a mixed
$p$-spin covariance $N\xi(R)$,
\begin{equation}\label{eq:OU-var-identity-appendix}
  \Var F_N(\beta)
  =\frac{\beta^2}{N}\int_0^\infty e^{-t}
    \E\bigl\langle\xi(R_{1,2})\bigr\rangle_t\,\dd t.
\end{equation}
For SK, $\xi(r)=r^2/2$, so
\begin{equation}\label{eq:OU-var-identity-appendix-SK}
  \Var F_N(\beta)
  =\frac{\beta^2}{2N}\int_0^\infty e^{-t}
    \E\bigl\langle R_{1,2}^2\bigr\rangle_t\,\dd t.
\end{equation}

\subsection{Incident-edge identity via Price's theorem and remainder bounds}
\label{subsec:ladder-prop}

\noindent\emph{Purpose.} We prove the direct $Q_e$ incident-edge identity used in the SK theorem.  The proof combines Price's Gaussian covariance-differentiation formula~\cite{Price1958,Chatterjee2014}, the zero-field spin-flip symmetry of the SK Hamiltonian~\cite{SherringtonKirkpatrick1975,Panchenko2013Book}, Schur-complement positivity~\cite[Ch.~0]{HornJohnson2013}, and the Hermite--Mehler monotonicity of the OU semigroup~\cite{Janson1997,Chatterjee2014}.  We also identify the exact finite-$N$ exclusion correction and record the sign and size information on the Schur/interpolation remainders.

\smallskip

\begin{lemma}[Replica parity at $h=0$]\label{lem:parity}
At $h=0$, for any disorder realization and any bounded observable $\mathcal O$
of the four replicas, if $\mathcal O$ is \emph{odd} under the global flip
$\sigma^{(a)}\mapsto -\sigma^{(a)}$ for at least one replica index $a\in\{1,2,3,4\}$,
then $\langle \mathcal O\rangle_t=0$.
\end{lemma}

\begin{proof}
The zero-field two-spin SK Hamiltonian $H^{(s)}(\sigma)=N^{-1/2}\sum_{i<j}J^{(s)}_{ij}\sigma_i\sigma_j$ is invariant under the global spin flip $\sigma\mapsto-\sigma$~\cite{SherringtonKirkpatrick1975,Panchenko2013Book}. Hence the Gibbs measure for each replica
is invariant under $\sigma^{(a)}\mapsto-\sigma^{(a)}$. If $\mathcal O$ changes sign
under this flip, $\langle\mathcal O\rangle_t=-\langle\mathcal O\rangle_t=0$.
\end{proof}

\paragraph{Coupling coordinates for the incident-edge identity.}
In the identity proof we use the off-diagonal coupling coordinates, which are the usual SK disorder variables~\cite{SherringtonKirkpatrick1975,Panchenko2013Book}.  By Lemma~\ref{lem:GOE-vs-offdiag}, removing the spin-independent GOE diagonal leaves every Gibbs expectation and overlap quantity unchanged. Its separate contribution to the free-energy variance is the explicit $O(N^{-2})$ shift stated there. Let
\[
  \mathcal E_N:=\{\{u,v\}:1\le u<v\le N\},
  \qquad
  I_j:=\{\{k,j\}:k\ne j\}
\]
be the edge set and the set of edges incident to~$j$. We write $J^1$ and $J^2$ for the two disorder vectors indexed by the graph edges in the coupled systems. Under the original OU pair, the two copies of each edge variable have cross-covariance $\rho_t$ and different edges are independent. We denote expectation for this full OU pair by $\E_{\rho_t}$. For $0\le \lambda\le\rho_t$, the notation $\E_{j,\lambda}$ means that the one-system marginals are unchanged, the cross-covariance is $\lambda$ on edges in $I_j$, remains $\rho_t$ on edges in $\mathcal E_N\setminus I_j$, and all cross-covariances between different edges are zero.

The next lemma applies the Hermite--Mehler expansion on Gaussian space~\cite{Janson1997,Chatterjee2014} while varying the correlations only in the coordinates in $I_j$.

\begin{lemma}[Partial-covariance Hermite monotonicity]
\label{lem:partial-hermite}
Fix $j$ and let $\gamma$ be the standard Gaussian law on $\mathbb R$. Let $f\in L^2(\gamma^{\otimes \mathcal E_N})$ have the orthonormal Hermite expansion
\[
  f(J)=\sum_{\alpha}a_\alpha H_\alpha(J),
  \qquad |\alpha|:=\sum_{e\in\mathcal E_N}\alpha_e,
  \qquad m_j(\alpha):=\sum_{e\in I_j}\alpha_e .
\]
Here $\alpha=(\alpha_e)_{e\in\mathcal E_N}$ ranges over finitely supported multi-indices and $H_\alpha$ is the tensor product of the one-dimensional orthonormal Hermite polynomials.
Then, for $0\le\lambda\le\rho_t$,
\begin{equation}\label{eq:partial-hermite-expansion}
  \E_{j,\lambda}[f(J^1)f(J^2)]
  =\sum_{\alpha}a_\alpha^2\,
  \rho_t^{|\alpha|-m_j(\alpha)}\lambda^{m_j(\alpha)} .
\end{equation}
Consequently $\lambda\mapsto\E_{j,\lambda}[f(J^1)f(J^2)]$ is nondecreasing on
$[0,\rho_t]$, and
\begin{equation}\label{eq:partial-hermite-defect}
\int_0^1\Bigl(
 \E_{j,\rho_t}[f(J^1)f(J^2)]-
 \E_{j,s\rho_t}[f(J^1)f(J^2)]\Bigr)\,ds
\le
\sum_\alpha a_\alpha^2\rho_t^{|\alpha|}m_j(\alpha).
\end{equation}
\end{lemma}

\begin{proof}
For a pair of standard Gaussians with correlation $r$, the Hermite--Mehler identity gives $\E[H_m(X)H_n(Y)]=\delta_{mn}r^m$ for orthonormal Hermite polynomials~\cite{Janson1997,Chatterjee2014}. Tensorizing over edges
and using the definition of $\E_{j,\lambda}$ gives
\eqref{eq:partial-hermite-expansion}. Since all coefficients in the resulting
series are nonnegative and $0\le\lambda\le\rho_t$, monotonicity follows. Finally,
termwise integration gives
\[
\int_0^1\Bigl(\rho_t^{|\alpha|}
 -\rho_t^{|\alpha|-m_j(\alpha)}(s\rho_t)^{m_j(\alpha)}\Bigr)ds
=\rho_t^{|\alpha|}\frac{m_j(\alpha)}{m_j(\alpha)+1}
\le \rho_t^{|\alpha|}m_j(\alpha),
\]
and summing over $\alpha$ proves~\eqref{eq:partial-hermite-defect}.
\end{proof}

The following identity is a local application of Price's theorem for Gaussian
inputs~\cite{Price1958}. The same Gaussian interpolation formula is reviewed
in~\cite[Section~2.1]{Chatterjee2014}.

\begin{lemma}[Incident-edge covariance identity via Price's theorem]\label{lem:incident-price}
Fix $j\in[N]$ and $i\ne j$. Let
$G_{uv}(J):=\langle\sigma_u\sigma_v\rangle_J$ denote the single-system
SK two-point function at $\beta=1$, $h=0$.
For $0\le\lambda\le\rho_t$, let $\E_{j,\lambda}$ be the partial-covariance
expectation defined above: the incident edges in $I_j$ have cross-covariance
$\lambda$, the non-incident edges have cross-covariance $\rho_t$, and the
one-system marginals are standard.
Define the Schur-complement response
\[
  \widehat G_{ik}^{(j)}(J)
  :=G_{ik}(J)-G_{ij}(J)G_{kj}(J).
\]
Since $G(J)=(\langle\sigma_u\sigma_v\rangle_J)_{u,v}$ is the (unit-diagonal)
correlation matrix of the spin vector, its principal Schur complement at
site~$j$, $(\widehat G_{ik}^{(j)})_{i,k\ne j}$, is positive semidefinite
and bounded above (in the positive semidefinite order) by the reduced matrix
$G^{\backslash j}=(G_{ik})_{i,k\ne j}$
(see~\cite[Ch.~0]{HornJohnson2013}). Equivalently
$0\preceq \widehat G^{(j)}\preceq G^{\backslash j}$, where $\preceq$ denotes positive semidefinite order.
Then
\begin{equation}\label{eq:price-identity}
  C_{ij}^e(t)
  :=\E_{\rho_t}\bigl[G_{ij}(J^1)G_{ij}(J^2)\bigr]
  =\frac{\rho_t}{N}\sum_{k\ne j}
  \int_0^1
  \E_{j,s\rho_t}
  \bigl[
    \widehat G_{ik}^{(j)}(J^1)\,
    \widehat G_{ik}^{(j)}(J^2)
  \bigr]\,ds.
\end{equation}
\end{lemma}

\begin{proof}
At $h=0$, $\langle\sigma_u\rangle_J=0$. With $q_i^e$ as defined in Lemma~\ref{lem:Qe-local-vertex},
\[
\E\langle q_i^e q_j^e\rangle_t
=\E_{\rho_t}[G_{ij}(J^1)G_{ij}(J^2)].
\]
The mixed $(2,4)$ cross terms vanish by replica independence and
zero-field symmetry.

For fixed~$j$, define
$\Phi_{ij}(\lambda):=\E_{j,\lambda}[G_{ij}(J^1)G_{ij}(J^2)]$.
Then $\Phi_{ij}(\rho_t)=C_{ij}^e(t)$.

\emph{Vanishing at $\lambda=0$.}
At $\lambda=0$, the incident edge vectors at~$j$ are independent
between systems. Using the zero-field spin-flip symmetry from Lemma~\ref{lem:parity}, the map
$(J^1_{kj})_{k\ne j}\mapsto -(J^1_{kj})_{k\ne j}$
is equivalent (in distribution, jointly with all other randomness) to the local
site flip $\sigma_j\mapsto-\sigma_j$ in system~$1$, which sends
$G_{ij}(J^1)\mapsto-G_{ij}(J^1)$. Since the system-$1$ incident vector is a centered
Gaussian independent of all other randomness at $\lambda=0$, the conditional
expectation of $G_{ij}(J^1)$ over the system-$1$ incident edges vanishes,
and $\Phi_{ij}(0)=0$.

\emph{Covariance differentiation.}
By the Gaussian covariance differentiation formula of
Price~\cite{Price1958} (see also~\cite[Section~2.1]{Chatterjee2014}), applied to
the cross-covariance parameter $\lambda$ while marginal variances are
held fixed,
\[
  \Phi_{ij}'(\lambda)
  =\sum_{k\ne j}
  \E_{j,\lambda}\!\left[
    \frac{\partial G_{ij}(J^1)}{\partial J^1_{kj}}
    \,
    \frac{\partial G_{ij}(J^2)}{\partial J^2_{kj}}
  \right].
\]
The single-system derivative is the Gibbs covariance derivative (see, e.g.,~\cite[Ch.~1]{Panchenko2013Book}): since
$(\sigma_i\sigma_j)(\sigma_k\sigma_j)=\sigma_i\sigma_k$ for Ising spins,
\begin{equation}\label{eq:single-system-deriv}
  \frac{\partial G_{ij}(J)}{\partial J_{kj}}
  =\frac{\beta}{\sqrt N}\,
  \Cov_J(\sigma_i\sigma_j,\,\sigma_k\sigma_j)
  =\frac{\beta}{\sqrt N}\,
  \widehat G_{ik}^{(j)}(J).
\end{equation}
At $\beta_c=1$ the prefactor is $1/\sqrt{N}$.
Substituting~\eqref{eq:single-system-deriv} into $\Phi_{ij}'(\lambda)$
and integrating from~$0$ to~$\rho_t$ with the change of variables
$\lambda=s\rho_t$ yields~\eqref{eq:price-identity}.
\end{proof}

Before stating the summed identity, we define the two per-pair remainders that appear in it. For $i\ne j$, write $G_{uv}^a=G_{uv}(J^a)$ and $\widehat G_{ik}^{(j),a}=\widehat G_{ik}^{(j)}(J^a)$, where the Schur-complement response $\widehat G^{(j)}$ was defined in Lemma~\ref{lem:incident-price}. Set
\begin{align}
  \mathcal R_{ij}^{\mathrm{Schur}}(t)
  &:=\frac{\rho_t}{N}\sum_{k\ne j}
  \int_0^1\E_{j,s\rho_t}\!\Bigl[
    \widehat G_{ik}^{(j),1}\widehat G_{ik}^{(j),2}
    -G_{ik}^1 G_{ik}^2
  \Bigr]ds,
  \label{eq:R-schur}
  \\
  \mathcal R_{ij}^{\mathrm{interp}}(t)
  &:=\frac{\rho_t}{N}\sum_{k\ne j}
  \int_0^1\Bigl(
    \E_{j,s\rho_t}[G_{ik}^1 G_{ik}^2]
    -\E_{\rho_t}[G_{ik}^1 G_{ik}^2]
  \Bigr)ds.
  \label{eq:R-interp}
\end{align}
The Schur remainder is the error made by replacing the full two-point matrix $G^{\backslash j}$ by the Schur-complement response at the fixed site $j$. The interpolation remainder is the error made by averaging the kernel term over partial incident-edge covariance $s\rho_t$ rather than evaluating it at the full OU covariance $\rho_t$.

The next proposition is the source of the identity used in the main text. The symbols with calligraphic $\mathcal R$ are per ordered pair $(i,j)$ with $i\ne j$. The symbols with fraktur $\mathfrak R$ are the normalized sums over all such pairs. This distinction is needed because the partial-covariance expectation depends on the fixed incident site $j$. The proposition is obtained from Lemma~\ref{lem:incident-price} by summing over vertices and keeping the finite-$N$ exclusions exactly. The Gaussian input is Price's covariance-differentiation formula~\cite{Price1958,Chatterjee2014}.

\begin{proposition}[Direct $Q_e$ incident-edge identity]\label{prop:Qe-direct-ladder}
At $\beta_c=1$ and $h=0$, set
\begin{equation}\label{eq:global-remainders}
  \mathfrak R_N^{\mathrm{Schur}}(t)
  :=\frac1N\sum_{j=1}^N\sum_{\substack{i=1\\ i\ne j}}^N
      \mathcal R_{ij}^{\mathrm{Schur}}(t),
  \qquad
  \mathfrak R_N^{\mathrm{interp}}(t)
  :=\frac1N\sum_{j=1}^N\sum_{\substack{i=1\\ i\ne j}}^N
      \mathcal R_{ij}^{\mathrm{interp}}(t).
\end{equation}
Then the OU overlap susceptibility $\chi_N(t)=N\,\E\langle Q_e^2\rangle_t$
satisfies the exact finite-$N$ identity
\begin{equation}\label{eq:Qe-exact-identity}
  \Bigl(1-\rho_t+\frac{2\rho_t}{N}\Bigr)\chi_N(t)
  =1+\frac{\rho_t}{N}
   +\mathfrak R_N^{\mathrm{Schur}}(t)
   +\mathfrak R_N^{\mathrm{interp}}(t),
\end{equation}
where the per-pair remainders are defined in
\eqref{eq:R-schur}--\eqref{eq:R-interp}.
\end{proposition}

\begin{proof}
The diagonal contribution to $N\E\langle Q_e^2\rangle_t$ is exactly $1$
(Lemma~\ref{lem:cross-only-vertex}). For $i\ne j$, the incident-edge identity
(Lemma~\ref{lem:incident-price}) gives the exact finite-$N$ identity
\[
  C_{ij}^e(t)
  =\frac{\rho_t}{N}\sum_{k\ne j}
  \int_0^1\E_{j,s\rho_t}\!\bigl[
    \widehat G_{ik}^{(j)}(J^1)\,
    \widehat G_{ik}^{(j)}(J^2)
  \bigr]\,ds.
\]

\smallskip\noindent\textbf{Schur/interpolation decomposition.}
In this proof $G_{uv}^a$ abbreviates $G_{uv}(J^a)$ for $a=1,2$. Expanding $\widehat G^{(j),1}_{ik}\,\widehat G^{(j),2}_{ik}
=(G_{ik}^1-G_{ij}^1G_{kj}^1)(G_{ik}^2-G_{ij}^2G_{kj}^2)$
produces four terms:
\begin{equation}\label{eq:schur-expand}
  \widehat G^{(j),1}_{ik}\,\widehat G^{(j),2}_{ik}
  = G_{ik}^1 G_{ik}^2
  - G_{ik}^1 G_{ij}^2 G_{kj}^2
  - G_{ik}^2 G_{ij}^1 G_{kj}^1
  + G_{ij}^1 G_{kj}^1 G_{ij}^2 G_{kj}^2.
\end{equation}
Hence
\[
  C_{ij}^e=\frac{\rho_t}{N}\sum_{k\ne j}C_{ik}^e
  +\mathcal R_{ij}^{\mathrm{Schur}}(t)
  +\mathcal R_{ij}^{\mathrm{interp}}(t),
\]
with the two remainders defined in~\eqref{eq:R-schur}--\eqref{eq:R-interp}. The Schur remainder $\mathcal R^{\mathrm{Schur}}$ contains the cubic and quartic correlation terms from~\eqref{eq:schur-expand}. The interpolation remainder $\mathcal R^{\mathrm{interp}}$ measures the difference between the expectation of the kernel term averaged over $s$ and its expectation at covariance $\rho_t$ on every edge. Both are exact terms of the identity, not truncation errors.

\smallskip\noindent\textbf{Summation with the fixed site excluded.}
Summing the exact identity over $i\ne j$ for fixed~$j$, then over~$j$,
and dividing by $N$: the kernel term $\frac{\rho_t}{N}\sum_{k\ne j}
C_{ik}^e$ excludes $k=j$. Write $T:=\sum_{i,k}C_{ik}^e=N\chi_N(t)$ and
$S_j:=\sum_i C_{ij}^e$. Then
$\sum_{i\ne j}\sum_{k\ne j}C_{ik}^e = T-2S_j+C_{jj}^e = T-2S_j+1$.
Summing over $j$ and dividing by $N^2$ gives
$\frac{\rho_t}{N^2}\sum_j(T-2S_j+1) = \rho_t(1-\frac{2}{N})\chi_N(t)+\frac{\rho_t}{N}$.
Adding the diagonal contribution $C_{jj}^e=1$ and the summed remainder, the
exact identity is~\eqref{eq:Qe-exact-identity}.
\end{proof}

\begin{remark}[Global remainder structure]\label{rem:global-remainder-structure}
The Schur part of the global remainder comes from summing the cubic and quartic terms in~\eqref{eq:schur-expand}. Suppressing superscripts and the fixed parameters $j,s$ only in this explanatory sentence, these terms have the schematic forms
\[
  \frac{\rho_t}{N^2}\sum_{i,j,k}
  \E_{j,s\rho_t}[G_{ik}G_{ij}G_{kj}]
\]
for the cubic contribution, and
\[
  \frac{\rho_t}{N^2}\sum_{i,j,k}
  \E_{j,s\rho_t}[G_{ij}G_{kj}G_{ij}G_{kj}]
\]
for the quartic contribution. The cubic term has the mixed-overlap structure isolated below. The quartic term has the favorable sign after global summation by the Schur product theorem. See Lemma~\ref{lem:remainder-signs}.
The interpolation remainder is controlled by Price's theorem~\cite{Price1958} applied to the incident edges, or equivalently by the Hermite--Mehler monotonicity in Lemma~\ref{lem:partial-hermite}, producing a second-order correction in the cross-covariance.
Lemma~\ref{lem:interp-integrated-app} controls the interpolation
remainder after integration. The integrated Schur bound
\eqref{eq:ass-SK-narrow} is proved in
Proposition~\ref{prop:cubic-from-moments} by estimating the mixed cubic
contribution. No pointwise interpolation bound of order $N^{-1}$ is
needed for that verification. The same proposition also bounds the
total integrated remainder directly through the exact identity.
The weaker integrated bound~\eqref{eq:Qe-integrated-remainder} is enough for the variance formula.
For the tail $t\ge1$, the denominator is uniformly bounded away from zero because $1-\rho_t\ge1-e^{-1}$.
\end{remark}

\begin{remark}[Global remainder notation and pointwise strengthening]
\label{rem:global-remainder-notation}
Write
\[
  \mathfrak R_N(t)
  :=\mathfrak R_N^{\mathrm{Schur}}(t)
    +\mathfrak R_N^{\mathrm{interp}}(t),
  \qquad
  \theta_N(t):=1-\rho_t+\frac{2\rho_t}{N}.
\]
The term $\rho_t/N$ in \eqref{eq:Qe-exact-identity} is an explicit
finite-$N$ numerator correction, not part of the signed remainder.  The
variance theorem requires only
\begin{equation}\label{eq:Qe-integrated-remainder}
  \int_{c/\chi_{\mathrm{SG}}(N)}^1 e^{-t}
  \frac{-\mathfrak R_N(t)}{\theta_N(t)}\,\dd t=O_c(1),
\end{equation}
which becomes Assumption~\ref{ass:SK-remainder-closure} after the sign lemma.
The stronger pointwise estimate
\begin{equation}\label{eq:Qe-direct-ladder}
  |\mathfrak R_N(t)|\le C(c)/N
  \qquad\text{for }t\in[c/\chi_{\mathrm{SG}}(N),1]
\end{equation}
would imply \eqref{eq:Qe-integrated-remainder} and
$\theta_N(t)\chi_N(t)=1+O(N^{-1})$. As shown in
Remark~\ref{cor:scalar-ladder}, it cannot hold uniformly on this
window at criticality for any fixed $c>1$. The main results require
only the integrated bound, verified in
Proposition~\ref{prop:cubic-from-moments}.  For $t\ge1$, Corollary~\ref{cor:upper-profile} gives a uniformly
bounded tail contribution.
\end{remark}

The Schur sign below uses Schur complements and the Schur product theorem~\cite{Schur1911,HornJohnson2013}. The interpolation sign uses the Hermite--Mehler expansion in Lemma~\ref{lem:partial-hermite}.

\begin{lemma}[Signs of the Schur and interpolation remainders]\label{lem:remainder-signs}
For every $t\ge 0$,
\[
  \mathfrak R_N^{\mathrm{Schur}}(t)\le 0,
  \qquad
  \mathfrak R_N^{\mathrm{interp}}(t)\le 0.
\]
\end{lemma}

\begin{proof}
\emph{Schur.} By Lemma~\ref{lem:incident-price}, for each disorder realization and fixed $j$,
\[
  0\preceq A_\ell:=\widehat G^{(j),\ell}\preceq B_\ell:=G^{\ell,\backslash j},
  \qquad \ell=1,2.
\]
Here $\circ$ denotes entrywise (Hadamard) product. Then
\[
  B_1\circ B_2-A_1\circ A_2
  =(B_1-A_1)\circ B_2+A_1\circ(B_2-A_2)\succeq0
\]
by the Schur product theorem for Hadamard products~\cite{Schur1911,HornJohnson2013}. Taking the all-ones quadratic form gives
\[
  \sum_{i,k\ne j}\widehat G_{ik}^{(j),1}\widehat G_{ik}^{(j),2}
  \le
  \sum_{i,k\ne j}G_{ik}^1G_{ik}^2.
\]
After summing over $j$ and integrating in $s$, this gives the global sign
$\mathfrak R_N^{\mathrm{Schur}}(t)\le0$. No pointwise sign for each individual per-pair remainder $\mathcal R_{ij}^{\mathrm{Schur}}$ is claimed.

\emph{Interpolation.}
Apply Lemma~\ref{lem:partial-hermite} with $f(J)=G_{ik}(J)$. Since
$0\le s\rho_t\le\rho_t$, the partial-covariance expectation
$\E_{j,s\rho_t}[G_{ik}(J^1)G_{ik}(J^2)]$ is no larger than the full OU
expectation $\E_{j,\rho_t}[G_{ik}(J^1)G_{ik}(J^2)]=\E_{\rho_t}[G_{ik}^1G_{ik}^2]$.
Thus every summand in~\eqref{eq:R-interp} is nonpositive, and
\eqref{eq:global-remainders} gives $\mathfrak R_N^{\mathrm{interp}}\le0$.
\end{proof}

\begin{corollary}[Unconditional pointwise upper bound]\label{cor:upper-profile}
For every $t\ge 0$ and every $N\ge 2$,
\begin{equation}\label{eq:upper-profile}
  \chi_N(t)\;\le\;\frac{1+\rho_t/N}{\theta_N(t)},
  \qquad
  \theta_N(t)=1-\rho_t+\frac{2\rho_t}{N}.
\end{equation}
\end{corollary}

\begin{proof}
Combine~\eqref{eq:Qe-exact-identity} with Lemma~\ref{lem:remainder-signs}
and divide by $\theta_N(t)>0$.
\end{proof}

\begin{remark}[Endpoint consistency bound]\label{rem:endpoint-sanity}
Setting $t=0$ in Corollary~\ref{cor:upper-profile} gives the finite-$N$
check
\[
  \chi_{\rm SG}(N)\le
  \frac{1+1/N}{2/N}=\frac{N+1}{2}.
\]
This slightly improves the trivial bound $\chi_{\rm SG}(N)\le N$ and is
consistent with the sign and normalization of the exact identity.
\end{remark}

\begin{lemma}[Pointwise bound for the interpolation remainder]\label{lem:interp-pointwise-app}
For every $t\ge 0$,
\begin{equation}\label{eq:interp-pointwise}
  0\;\ge\;\mathfrak R_N^{\mathrm{interp}}(t)
  \;\ge\; -\,\frac{2\rho_t^{2}}{N}\,
  \frac{\partial \widehat\chi_N}{\partial \rho}\bigg|_{\rho=\rho_t},
\end{equation}
where $\widehat\chi_N(\rho)$ denotes the same coupled susceptibility viewed as a function of the Hamiltonian correlation parameter $\rho$ rather than of OU time $t$. Thus $\chi_N(t)=\widehat\chi_N(e^{-t})$, and
\[
  \widehat\chi_N(\rho):=N^{-1}\sum_{i,k}\sum_\alpha
  a_{ik,\alpha}^{2}\,\rho^{|\alpha|}.
\]
Here $a_{ik,\alpha}$ are the Hermite coefficients of the two-point function $G_{ik}$ as a function of the Gaussian disorder. This is the Hermite-series representation of $\chi_N$ as a function of the OU correlation parameter~\cite{Janson1997,Chatterjee2014}.
\end{lemma}

\begin{proof}
Apply the defect bound~\eqref{eq:partial-hermite-defect} of
Lemma~\ref{lem:partial-hermite} to $f(J)=G_{ik}(J)$:
\[
  \int_0^1\bigl(\E_{\rho_t}[G_{ik}^{1}G_{ik}^{2}]-
   \E_{j,s\rho_t}[G_{ik}^{1}G_{ik}^{2}]\bigr)\,ds
  \le \sum_\alpha a_{ik,\alpha}^{2}\rho_t^{|\alpha|}m_j(\alpha).
\]
Summing over~$j$, each edge contributes to exactly two endpoint sets, so
$\sum_{j=1}^{N}m_j(\alpha)=2|\alpha|$. Then summing over $i,k$ (extending the
sums over $i\ne j$ and $k\ne j$ to all pairs only adds nonnegative defect
terms) and using the
factor $\rho_t/N^2$ in the global definition of
$\mathfrak R_N^{\mathrm{interp}}$ gives
\[
  -\,\mathfrak R_N^{\mathrm{interp}}(t)
   \le \frac{2\rho_t}{N^{2}}\sum_{i,k}\sum_\alpha
       |\alpha|\,a_{ik,\alpha}^{2}\rho_t^{|\alpha|}
   =\frac{2\rho_t^{2}}{N}\,
       \frac{\partial \widehat\chi_N}{\partial \rho}(\rho_t),
\]
because $\widehat\chi_N(\rho)=N^{-1}\sum_{i,k}\sum_\alpha
 a_{ik,\alpha}^{2}\rho^{|\alpha|}$. The inequality
$\mathfrak R_N^{\mathrm{interp}}\le0$ is Lemma~\ref{lem:remainder-signs}.
\end{proof}

\begin{lemma}[Integrated bound for the interpolation remainder]\label{lem:interp-integrated-app}
Unconditionally, for every fixed $c>0$,
\begin{equation}\label{eq:interp-integrated}
  \int_{c/\chi_{\mathrm{SG}}(N)}^{1}
    e^{-t}\,\frac{|\mathfrak R_N^{\mathrm{interp}}(t)|}{\theta_N(t)}\,\dd t
  \;\le\;\frac{C_c}{N}\,\chi_{\mathrm{SG}}(N)^{2}.
\end{equation}
\end{lemma}

\begin{proof}
For $N=1$, the interpolation remainder vanishes identically.
Assume $N\ge2$.
If $c/\chi_{\mathrm{SG}}(N)\ge1$, the claim is immediate.  Otherwise,
change variables to $\rho=e^{-t}$, $\dd t=-\dd\rho/\rho$,
and write $\rho_0:=e^{-c/\chi_{\mathrm{SG}}(N)}$, $\rho_1:=e^{-1}$, and $\theta_N(\rho):=1-\rho+2\rho/N$.
By Lemma~\ref{lem:interp-pointwise-app},
\[
  \int_{c/\chi_{\mathrm{SG}}(N)}^{1}
   e^{-t}\,\frac{|\mathfrak R_N^{\mathrm{interp}}(t)|}{\theta_N}\,\dd t
   \;\le\;\frac{2}{N}
     \int_{\rho_1}^{\rho_0}
       \frac{\rho^2\,\partial_\rho\widehat\chi_N(\rho)}{\theta_N(\rho)}\,\dd\rho
   \;\le\;\frac{2}{N}
     \int_{\rho_1}^{\rho_0}
       \frac{\partial_\rho\widehat\chi_N(\rho)}{\theta_N(\rho)}\,\dd\rho,
\]
using $\rho^2\le 1$ on $[\rho_1,\rho_0]$ and $\partial_\rho\widehat\chi_N\ge 0$, which follows from the nonnegative squared Hermite coefficients in the OU expansion~\cite{Janson1997,Chatterjee2014}. By Corollary~\ref{cor:upper-profile},
$\widehat\chi_N(\rho)\le C/\theta_N(\rho)$ unconditionally, and since
$\theta_N(\rho)=1-\rho(1-2/N)$ is decreasing in~$\rho$,
$\partial_\rho(1/\theta_N(\rho))\ge 0$. Integration by parts gives
\[
  \int_{\rho_1}^{\rho_0}
     \frac{\partial_\rho\widehat\chi_N(\rho)}{\theta_N(\rho)}\,\dd\rho
   =\Big[\frac{\widehat\chi_N(\rho)}{\theta_N(\rho)}\Big]_{\rho_1}^{\rho_0}
     -\int_{\rho_1}^{\rho_0}
        \widehat\chi_N(\rho)\,\partial_\rho\!\Big(\frac{1}{\theta_N(\rho)}\Big)\dd\rho.
\]
The boundary term at $\rho_0$ is bounded by
$C/\theta_N(\rho_0)^{2}\le C_c\,\chi_{\mathrm{SG}}(N)^{2}$, since
$\theta_N(\rho_0)\asymp c/\chi_{\mathrm{SG}}(N)+1/N$.
The boundary term at $\rho_1$ is $O(1)$. The integral is
nonnegative and is subtracted, hence drops out of the upper bound.
This yields~\eqref{eq:interp-integrated}.
\end{proof}

\begin{remark}[Reduction of the integrated problem to the Schur contribution]\label{rem:interp-integrated-strength}
Lemma~\ref{lem:interp-integrated-app} provides an unconditional
bound on the integrated interpolation remainder of order
$\chi_{\mathrm{SG}}(N)^{2}/N$. In any regime where
$\chi_{\mathrm{SG}}(N)=O(N^{\gamma})$ with $\gamma\le \tfrac12$,
this contribution is $O(1)$. Under the critical scaling
$\gamma=1/3$, proved in~\cite{DuHuang2026},
it is $O(N^{-1/3})=o(1)$. Consequently, verification through the separate
contributions reduces the integrated-remainder estimate
\eqref{eq:Qe-integrated-remainder} to the Schur contribution,
which is bounded in Proposition~\ref{prop:cubic-from-moments}.  Lemma~\ref{lem:interp-integrated-app} does not give
an $O(N^{-1})$ pointwise bound on the interpolation remainder, so no
analogous reduction is claimed for the stronger pointwise
estimate~\eqref{eq:Qe-direct-ladder}.
\end{remark}

\subsection{Schur estimate and mixed cubic control}

The estimates in this subsection use two external inputs, both cited here because they are separate from the new finite-$N$ summation argument: Schur-complement and Hadamard-product positivity for correlation matrices~\cite{Schur1911,HornJohnson2013}, and the replica notation for rewriting products of Gibbs averages as Gibbs averages over independent replicas~\cite[Ch.~1]{Panchenko2013Book}.

\begin{lemma}[Fixed-site representation of the cubic Schur term]%
\label{lem:triple-overlap-rep}
For fixed incident site $j$, replica $(1)$ from system~$1$, and independent
replicas $(2),(3)$ from system~$2$, set
$q_j^{23}:=\sigma_j^{(2)}\sigma_j^{(3)}$. Then
\begin{equation}\label{eq:cubic-overlap}
  \frac{1}{N^2}\sum_{u,v=1}^N
  G_{uv}^1\,G_{uj}^2\,G_{vj}^2
  =\bigl\langle R_{12}R_{13}q_j^{23}\bigr\rangle .
\end{equation}
Consequently the incident-edge Schur cubic is measured by
\begin{equation}\label{eq:mixed-cubic-functional}
  \mathfrak C_N(t):=\frac1N\sum_{j=1}^N\int_0^1
  \E_{j,s\rho_t}
  \left|\bigl\langle R_{12}R_{13}q_j^{23}\bigr\rangle\right|\,ds .
\end{equation}
\end{lemma}

\begin{proof}
Using the replica convention for products of Gibbs averages~\cite[Ch.~1]{Panchenko2013Book}, take two independent replicas $(2),(3)$ of the second Gibbs measure:
\[
  G_{uj}^2G_{vj}^2
  =\bigl\langle
    \sigma_u^{(2)}\sigma_j^{(2)}
    \sigma_v^{(3)}\sigma_j^{(3)}
  \bigr\rangle_{J^2}.
\]
Multiplying by $G_{uv}^1$ and summing over $u,v$ gives
\[
  \sum_{u,v}G_{uv}^1G_{uj}^2G_{vj}^2
  =N^2\bigl\langle R_{12}R_{13}q_j^{23}\bigr\rangle,
\]
which is~\eqref{eq:cubic-overlap}. The fixed-site form is essential under
the incident-edge interpolation $\E_{j,s\rho_t}$, because the law itself depends
on~$j$.
\end{proof}

\begin{remark}[Overlap representation of the Schur remainder]\label{rem:schur-overlap}
If the ensemble did not depend on $j$, averaging~\eqref{eq:cubic-overlap} over
$j$ would replace $q_j^{23}$ by the global overlap $R_{23}$. In the actual
incident-edge interpolation, however, the expectation $\E_{j,s\rho_t}$ depends
on the fixed incident site~$j$, so the appropriate object is the fixed-site functional
$\mathfrak C_N(t)$ in~\eqref{eq:mixed-cubic-functional}. This is why
Proposition~\ref{prop:mixed-cubic-schur} is formulated in terms of
$q_j^{23}$ rather than only $R_{23}$.
\end{remark}

\begin{lemma}[Positivity and endpoint correction for the cubic sums]
\label{lem:schur-endpoint-bookkeeping}
At $\beta=1$, $h=0$, and $N\ge2$, define the restricted and unrestricted
cubic sums
\[
\begin{aligned}
  X_{j,1}^{\mathrm{res}}
  &:=\sum_{\substack{u\ne j\\v\ne j}}G_{uv}^1G_{uj}^2G_{vj}^2,
  &\qquad
  X_{j,1}
  &:=\sum_{u,v=1}^NG_{uv}^1G_{uj}^2G_{vj}^2,\\
  X_{j,2}^{\mathrm{res}}
  &:=\sum_{\substack{u\ne j\\v\ne j}}G_{uv}^2G_{uj}^1G_{vj}^1,
  &
  X_{j,2}
  &:=\sum_{u,v=1}^NG_{uv}^2G_{uj}^1G_{vj}^1 .
\end{aligned}
\]
For each realization all four sums are nonnegative. Moreover, for
$0\le\lambda\le\rho_t$ and $\ell=1,2$,
\begin{equation}\label{eq:schur-endpoint-error}
  \E_{j,\lambda}X_{j,\ell}^{\mathrm{res}}
  \le\E_{j,\lambda}X_{j,\ell}-1.
\end{equation}
\end{lemma}

\begin{proof}
Write $v=(G_{uj}^2)_{u=1}^N$. Then
$X_{j,1}=v^{\mathsf T}G^1v\ge0$, since $G^1$ is positive semidefinite.
The restricted sum is the corresponding quadratic form of the principal
submatrix obtained by deleting site $j$. Exchanging the systems proves
the same facts for $X_{j,2}$ and $X_{j,2}^{\mathrm{res}}$.
Inclusion--exclusion and $G_{jj}^1=G_{jj}^2=1$ give, for $\ell=1,2$,
\begin{equation}\label{eq:first-cubic-endpoint}
  X_{j,\ell}-X_{j,\ell}^{\mathrm{res}}
  =2\sum_{u=1}^NG_{uj}^1G_{uj}^2-1.
\end{equation}
By Lemma~\ref{lem:partial-hermite}, each expectation
$\E_{j,\lambda}[G_{uj}^1G_{uj}^2]$ is nonnegative. The diagonal term
equals $1$, so the expectation of the right-hand side is at least $1$.
This proves~\eqref{eq:schur-endpoint-error}.
\end{proof}

\begin{proposition}[Schur contribution controlled by the fixed-site cubic term]
\label{prop:schur-cubic-control}
At $\beta=1$, $h=0$, and $N\ge2$, for every $t\ge0$,
\begin{equation}\label{eq:schur-cubic-control}
  0\le-\mathfrak R_N^{\mathrm{Schur}}(t)
  \le2\rho_tN\mathfrak C_N(t)-\frac{2\rho_t}{N}
  \le2\rho_tN\mathfrak C_N(t).
\end{equation}
\end{proposition}

\begin{proof}
Expanding the Schur complements gives the exact identity
\begin{equation}\label{eq:schur-cubic-exact-sum}
\begin{split}
  -\mathfrak R_N^{\mathrm{Schur}}(t)
  =\frac{\rho_t}{N^2}\sum_{j=1}^N\int_0^1
  \E_{j,s\rho_t}\left[
    X_{j,1}^{\mathrm{res}}+X_{j,2}^{\mathrm{res}}
    -\left(\sum_{u\ne j}G_{uj}^1G_{uj}^2\right)^2
  \right]\,\dd s.
\end{split}
\end{equation}
Dropping the nonnegative square and applying
Lemma~\ref{lem:schur-endpoint-bookkeeping} yield
\[
  -\mathfrak R_N^{\mathrm{Schur}}(t)
  \le\frac{\rho_t}{N^2}\sum_{j=1}^N\int_0^1
  \E_{j,s\rho_t}[X_{j,1}+X_{j,2}]\,\dd s-\frac{2\rho_t}{N}.
\]
Since $X_{j,1}\ge0$, the absolute value in the definition of
$\mathfrak C_N(t)$ is redundant. Lemma~\ref{lem:triple-overlap-rep}
shows that the term containing $X_{j,1}$ contributes $\rho_tN\mathfrak C_N(t)$.
The partial-covariance law is invariant under exchanging the two
systems, so the term containing $X_{j,2}$ has the same contribution.
The lower bound follows from Lemma~\ref{lem:remainder-signs}.
\end{proof}

\begin{remark}[Scope of Proposition~\ref{prop:schur-cubic-control}]
\label{rem:schur-cubic-scope}
Proposition~\ref{prop:schur-cubic-control} converts a separate estimate on
the fixed-site functional $\mathfrak C_N(t)$ into a bound on the Schur
remainder. The additional comparison and critical inputs in
Proposition~\ref{prop:cubic-from-moments} prove the integrated estimate
\eqref{eq:ass-SK-narrow}.  The exact incident-edge
identity and the unconditional variance upper bound do not depend on
Lemma~\ref{lem:schur-endpoint-bookkeeping} or on this proposition.
\end{remark}

\begin{proposition}[Reverse comparison with the cubic functional]
\label{prop:reverse-schur}
Put $A_N(t)=\rho_tN\mathfrak C_N(t)$ and
$S_N(t)=-\mathfrak R_N^{\rm Schur}(t)$. Then
\begin{equation}\label{eq:reverse-schur}
 A_N(t)-E_N(t)\le S_N(t)\le2\bigl(A_N(t)-E_N(t)\bigr),
\end{equation}
where
\begin{equation}\label{eq:cubic-endpoint-sharp}
\begin{split}
 E_N(t)&=\frac{\rho_t}{N^2}\sum_j\int_0^1
 \E_{j,s\rho_t}\left[2\sum_uG_{uj}^1G_{uj}^2-1\right]\dd s,\\
 \frac{\rho_t}{N}&\le E_N(t)
 \le\frac{\rho_t}{N}\bigl(2\chi_N(t)-1\bigr).
\end{split}
\end{equation}
\end{proposition}
\begin{proof}
Let $B_a=(G_{uv}^a)_{u,v\ne j}$ and $v_a=(G_{uj}^a)_{u\ne j}$.
Schur positivity gives $B_a\succeq v_av_a^{\mathsf T}$, so
$X_{j,1}^{\rm res},X_{j,2}^{\rm res}\ge(v_1^{\mathsf T}v_2)^2=:Q_j$.
Consequently
\[
 \tfrac12(X_{j,1}^{\rm res}+X_{j,2}^{\rm res})
 \le X_{j,1}^{\rm res}+X_{j,2}^{\rm res}-Q_j
 \le X_{j,1}^{\rm res}+X_{j,2}^{\rm res}.
\]
Insert this into \eqref{eq:schur-cubic-exact-sum}, use
\eqref{eq:first-cubic-endpoint}, and exchange the systems to obtain
\eqref{eq:reverse-schur}. Lemma~\ref{lem:partial-hermite} bounds each
partial-covariance expectation of $G_{uj}^1G_{uj}^2$ between zero and
its value at covariance $\rho_t$ on every edge. The diagonal is one and the sum of these terms is $N\chi_N(t)$, proving
\eqref{eq:cubic-endpoint-sharp}.
\end{proof}

\begin{proposition}[Integrated cubic bound from critical moment estimates]
\label{prop:cubic-from-moments}
At $\beta=1$, let $m_4=\E\langle R_{12}^4\rangle$ and
$\chi=\chi_{\rm SG}(N)$. The cavity estimate
of~\cite[Proposition~5.7]{DuHuang2026} implies
\begin{equation}\label{eq:cubic-moment-profile}
 \rho_tN\mathfrak C_N(t)
 \le\frac{2\chi}{N}+C\sqrt{N\chi_N(t)m_4}+CNm_4.
\end{equation}
Together with the fourth-moment bound
of~\cite[Proposition~5.1]{DuHuang2026}, this proves
\eqref{eq:integrated-cubic-control} for every fixed $c>1$.
These estimates also verify Assumption~\ref{ass:SK-remainder-closure}
directly.
\end{proposition}
\begin{proof}
In our correlation convention, the cited cavity estimate is
\[
 (1-\rho_t)\chi_N(t)
 \ge1-\frac{2\chi}{N}-C\sqrt{N\chi_N(t)m_4}-CNm_4.
\]
Since $\chi_N(t)\ge1$, the exact identity and remainder signs give
\[
 \begin{aligned}
 0\le{}&-\mathfrak R_N^{\rm Schur}(t)-\mathfrak R_N^{\rm interp}(t)\\
 ={}&1-(1-\rho_t)\chi_N(t)
       -\frac{\rho_t}{N}\bigl(2\chi_N(t)-1\bigr)\\
 \le{}&1-(1-\rho_t)\chi_N(t).
 \end{aligned}
\]
The cited cavity estimate therefore bounds the total remainder by the
right-hand side of \eqref{eq:cubic-moment-profile} directly.
For the cubic functional, the identity and remainder signs give
\[
 S_N(t)\le1+\rho_t/N-\theta_N(t)\chi_N(t).
\]
Use $A_N\le S_N+E_N$ and the upper bound in
\eqref{eq:cubic-endpoint-sharp}. Since
$\theta_N=1-\rho_t+2\rho_t/N$, the finite-$N$ corrections cancel,
giving \eqref{eq:cubic-moment-profile}.
The quoted fourth-moment bound is $m_4=O(N^{-4/3})$ and entails
$\chi\le N\sqrt{m_4}=O(N^{1/3})$. Hence
\[
 \rho_tN\mathfrak C_N(t)
 \le C\left(N^{-1/6}\sqrt{\chi_N(t)}+N^{-1/3}\right).
\]
On $[c/\chi,1]$ use $\chi_N(t)\lesssim t^{-1}$ and
$\theta_N(t)\asymp t$. The required integral is at most
\[
 C_c\left(N^{-1/6}\sqrt\chi+N^{-1/3}\log\chi\right)=O_c(1).
\]
The same fourth-moment and integral estimates apply to the total
remainder, proving Assumption~\ref{ass:SK-remainder-closure} directly.
Empty windows contribute zero. The cavity input is taken from
\cite{DeyKang2026}, as explained in~\cite[Section~5.2]{DuHuang2026}.
Its interpolation need not be identified with our partial-covariance
law. For the cubic functional, Proposition~\ref{prop:reverse-schur}
supplies the comparison.
\end{proof}

\begin{lemma}[Baseline size of the mixed cubic functional]
\label{lem:mixed-cubic-baseline}
For every $t\ge0$,
\begin{equation}\label{eq:mixed-cubic-baseline}
  \mathfrak C_N(t)\le \frac{\chi_N(t)}{N}.
\end{equation}
\end{lemma}

\begin{proof}
For fixed $j$ and $s$, applying Cauchy--Schwarz under the product Gibbs measure gives
\[
  \left|\bigl\langle R_{12}R_{13}q_j^{23}\bigr\rangle\right|
  \le \bigl\langle R_{12}^2\bigr\rangle^{1/2}
       \bigl\langle R_{13}^2\bigr\rangle^{1/2}.
\]
Replicas $2$ and $3$ are exchangeable independent samples from the second
system, so the two second moments are equal for fixed disorder. The right-hand
side is $\langle R_{12}^2\rangle$. Taking $\E_{j,s\rho_t}$ and expanding,
\[
  \E_{j,s\rho_t}\langle R_{12}^2\rangle
  =\frac1{N^2}\sum_{u,v}
    \E_{j,s\rho_t}[G_{uv}(J^1)G_{uv}(J^2)] .
\]
Lemma~\ref{lem:partial-hermite} with $f=G_{uv}$ bounds each summand by the corresponding summand with covariance $\rho_t$ on every edge. Therefore
\[
  \E_{j,s\rho_t}\langle R_{12}^2\rangle
  \le \frac1{N^2}\sum_{u,v}\E[G_{uv}(J^1)G_{uv}(J^2)]
  =\frac{\chi_N(t)}{N}.
\]
Averaging over $j$ and $s$ proves~\eqref{eq:mixed-cubic-baseline}.
\end{proof}

\begin{corollary}[Crude Schur bound via the fixed-site cubic term]%
\label{lem:schur-overlap-bound}
At $\beta_c=1$, $h=0$, for each $t\ge0$,
\begin{equation}\label{eq:schur-loss-bound}
  |\mathfrak R_N^{\mathrm{Schur}}(t)|
  \le 2\rho_t\chi_N(t).
\end{equation}
\end{corollary}

\begin{proof}
Combine Proposition~\ref{prop:schur-cubic-control} with
Lemma~\ref{lem:mixed-cubic-baseline}, and discard the nonpositive term $-2\rho_t/N$.
\end{proof}

\begin{remark}[Comparison with the cubic criterion]\label{rem:remainder-gap}
Corollary~\ref{lem:schur-overlap-bound} is deliberately crude: it gives only
$|\mathfrak R_N^{\mathrm{Schur}}|\lesssim\rho_t(1+\chi_N)$, whereas the sufficient pointwise condition \eqref{eq:mixed-cubic-control} gains the factor $(\chi_{\rm SG}/N)^{1/2}$ and implies the integrated criterion.  When $\chi_{\rm SG}\asymp N^{1/3}$, the required gain
$(\chi_{\rm SG}/N)^{1/2}\asymp N^{-1/3}$ is precisely the root-mean-square
scale of a same-system overlap.  Thus the condition asks that the fixed-site
factor acquire the global overlap scale after averaging over the incident site
and the partial-covariance interpolation.
\end{remark}

\section{High-temperature validation of the incident-edge identity}
\label{app:high-temperature}

The critical argument uses $\beta=1$, but the incident-edge computation
tracks $\beta$ without structural change.  This appendix records the
general-$\beta$ identity and checks the remainder decomposition in the
classical high-temperature phase, where the limiting free-energy variance is
known~\cite{AizenmanLebowitzRuelle1987,CometsNeveu1995,GuerraToninelli2002CLT}.

For the remainder of this appendix, $\langle\cdot\rangle_t$ is the OU-coupled
product Gibbs expectation at inverse temperature $\beta$, and
\[
  \chi_N(t;\beta):=N\E\langle R_{1,2}^2\rangle_t,
  \qquad \chi_{\rm SG}(N;\beta):=\chi_N(0;\beta).
\]
When the Hamiltonian correlation itself is the variable, we write
$\widehat\chi_N(\rho;\beta)$, so that
$\chi_N(t;\beta)=\widehat\chi_N(e^{-t};\beta)$.
Define the general-$\beta$ per-pair remainders by
\eqref{eq:R-schur}--\eqref{eq:R-interp}, with all two-point functions taken at
inverse temperature $\beta$ and the prefactor $\rho_t/N$ replaced by
$\beta^2\rho_t/N$.  Their global sums are denoted
$\mathfrak R_N^{\rm Schur}(t;\beta)$ and
$\mathfrak R_N^{\rm interp}(t;\beta)$, and
$\mathfrak R_N(t;\beta)$ is their sum.

\begin{proposition}[Incident-edge identity at general $\beta$]
\label{prop:ht-identity}
For every $\beta>0$, $N\ge2$, and $t\ge0$,
\begin{equation}\label{eq:ht-identity}
  \theta_N^{(\beta)}(t)\chi_N(t;\beta)
  =1+\frac{\beta^2\rho_t}{N}
   +\mathfrak R_N^{\rm Schur}(t;\beta)
   +\mathfrak R_N^{\rm interp}(t;\beta),
  \qquad
  \theta_N^{(\beta)}(t)
  :=1-\beta^2\rho_t\left(1-\frac2N\right).
\end{equation}
Both remainders are nonpositive.  Therefore, whenever
$\theta_N^{(\beta)}(t)>0$,
\begin{equation}\label{eq:ht-pointwise}
  \chi_N(t;\beta)
  \le\frac{1+\beta^2\rho_t/N}
           {1-\beta^2\rho_t(1-2/N)}.
\end{equation}
Positivity is automatic for $0<\beta\le1$.  For $N>2$ and $\beta>1$ it
holds when $\rho_t<[\beta^2(1-2/N)]^{-1}$. For $N=2$ one has
$\theta_2^{(\beta)}(t)\equiv1$.
\end{proposition}

\begin{proof}
The Gibbs derivative \eqref{eq:single-system-deriv} acquires a factor $\beta$,
so Price's theorem contributes $\beta^2\rho_t/N$ instead of $\rho_t/N$.
The exclusion accounting in Proposition~\ref{prop:Qe-direct-ladder} then
produces $\beta^2\rho_t(1-2/N)\chi_N(t;\beta)$ and the numerator correction
$\beta^2\rho_t/N$.  The zero-field parity, Schur-complement, Schur-product,
and Hermite--Mehler arguments are independent of $\beta$, so the remainder
signs are unchanged.  Division by the denominator gives
\eqref{eq:ht-pointwise} precisely where it is positive.
\end{proof}

\begin{theorem}[Variance and susceptibility at every temperature]
\label{prop:finite-beta-bound}
For every $N\ge1$, $\beta>0$, and $h=0$, put
$q=\beta^2$ and $\chi=\chi_{\rm SG}(N;\beta)$. Then
\begin{align}
  2N^2\Var F_N(\beta)
    &\le\log\chi+1+(q-1)\chi,\label{eq:finite-beta-goe}\\
  2N^2\Var F_N^{\rm off}(\beta)
    &\le\log\chi+(q-1)(\chi-1).\label{eq:finite-beta-off}
\end{align}
For $0<\beta<1$, one also has the exact finite-size bounds
\begin{equation}\label{eq:ht-sharp-finite}
  2N^2\Var F_N(\beta)\le-\log(1-\beta^2),
  \qquad
  2N^2\Var F_N^{\rm off}(\beta)
    \le-\log(1-\beta^2)-\beta^2.
\end{equation}
\end{theorem}

\begin{proof}
Suppose $N\ge2$. For $qr<1$, Proposition~\ref{prop:ht-identity} gives
\begin{equation}\label{eq:simple-beta-profile}
  \widehat\chi_N(r;\beta)
  \le\frac{1+qr/N}{1-qr+2qr/N}\le\frac1{1-qr}.
\end{equation}
The second inequality follows since the difference of the last two
expressions, in the order right minus left, equals
\[
  \frac{qr(1+qr)}{N(1-qr)(1-qr+2qr/N)}\ge0.
\]
The Hermite argument of Lemma~\ref{lem:chi-CS} applies at every
$\beta$ and gives $\widehat\chi_N(r;\beta)\le\chi$.
Set $r_*=(1-\chi^{-1})/q$. Since $\chi\ge1$, $r_*\ge0$.
If $q\ge1$, then $r_*<1$. If $q<1$,~\eqref{eq:simple-beta-profile}
at $r=1$ gives $\chi\le(1-q)^{-1}$, hence again $r_*\le1$.
On $[0,r_*]$ we have $qr\le1-\chi^{-1}<1$, so no division by a
nonpositive denominator occurs, including when $\beta>1$.
Integrating the two bounds on their respective intervals yields
\[
\begin{split}
  2N^2\Var F_N(\beta)
  &=q\int_0^1\widehat\chi_N(r;\beta)\,\dd r\\
  &\le q\int_0^{r_*}\frac{\dd r}{1-qr}+q\chi(1-r_*)
   =\log\chi+1+(q-1)\chi.
\end{split}
\]
Subtracting the GOE diagonal variance proves~\eqref{eq:finite-beta-off}.
When $q<1$, integrating~\eqref{eq:simple-beta-profile} on the whole
interval $[0,1]$ gives~\eqref{eq:ht-sharp-finite} directly.
For $N=1$, $\chi=1$, $\Var F_1(\beta)=q/2$, and
$\Var F_1^{\rm off}(\beta)=0$. All claims follow, using
$q\le-\log(1-q)$ when $q<1$.
\end{proof}

\begin{remark}[A sufficient susceptibility bound near criticality]
\label{rem:both-temperature-sides}
If $\beta_N^2-1=O(N^{-1/3})$ and
$\chi_{\rm SG}(N;\beta_N)=O(N^{1/3})$, then
\eqref{eq:finite-beta-off} gives
$\Var F_N^{\rm off}(\beta_N)\le\tfrac16(\log N)/N^2+O(N^{-2})$
on either side of criticality. This implication is unconditional as
an inequality. The required susceptibility estimate throughout that
window is an additional input, not proved here. Proposition~\ref{prop:uniform-window-variance}
obtains the variance conclusion by a separate entropy argument.
The full limiting overlap law in that window is discussed
in~\cite[Section~1.4]{DuHuang2026Overlap}.
\end{remark}

For $0<\beta<1$ define the remainder integrated over $t\ge0$
\begin{equation}\label{eq:ht-remainder}
  D_N(\beta):=\int_0^\infty e^{-t}
  \frac{-\mathfrak R_N(t;\beta)}{\theta_N^{(\beta)}(t)}\,\dd t\ge0.
\end{equation}
Here $\theta_N^{(\beta)}(t)\ge1-\beta^2$, so no cutoff is needed.

\begin{theorem}[High-temperature variance identity and bounds]
\label{thm:ht-variance}
For $N\ge4$ and $0<\beta<1$, the following statements hold
unconditionally.  One may take the absolute constant below to be $C_0=9$.
\begin{enumerate}[label=\textup{(\roman*)},leftmargin=2.5em]
\item
\[
  \sup_{t\ge0}\chi_N(t;\beta)\le\frac{2}{1-\beta^2}.
\]
\item
\begin{equation}\label{eq:ht-remainder-identity}
  2N^2\Var F_N(\beta)
  =\log\frac1{1-\beta^2}-\beta^2D_N(\beta)+\mathcal E_N(\beta),
  \qquad
  |\mathcal E_N(\beta)|\le\frac{C_0}{N(1-\beta^2)}.
\end{equation}
\item
\begin{equation}\label{eq:ht-upper}
  2N^2\Var F_N(\beta)
  \le\log\frac1{1-\beta^2}.
\end{equation}
\item
\begin{equation}\label{eq:ht-remainder-bounds}
  0\le D_N(\beta)\le\frac{2}{1-\beta^2},
  \qquad
  \int_0^\infty e^{-t}
  \frac{-\mathfrak R_N^{\rm interp}(t;\beta)}
       {\theta_N^{(\beta)}(t)}\,\dd t
  \le\frac{4\beta^2}{N(1-\beta^2)^2}.
\end{equation}
\end{enumerate}
Thus the analog with integration over $t\ge0$ of Assumption~\ref{ass:SK-remainder-closure} is
unconditionally true for every fixed $\beta<1$.
\end{theorem}

\begin{proof}
Write $b=\beta^2(1-2/N)$ and use
$e^{-t}\,\dd t=-\dd\rho$ with $\rho=e^{-t}$, reversing the integration
limits.  The pointwise bound gives
\[
  \chi_N(t;\beta)
  \le\frac{1+\beta^2/N}{1-b}
  \le\frac{2}{1-\beta^2},
\]
which proves (i).

From \eqref{eq:ht-identity} and Theorem~\ref{thm:var-identity},
\begin{equation}\label{eq:ht-split}
  2N^2\Var F_N(\beta)
  =\beta^2\int_0^1\frac{1+\beta^2\rho/N}{1-b\rho}\,\dd\rho
   -\beta^2D_N(\beta).
\end{equation}
The main part is
\[
  \frac{1}{1-2/N}\log\frac1{1-b}.
\]
Using
$1-b=1-\beta^2+2\beta^2/N$,
$(1-2/N)^{-1}\le1+4/N$, and
$\log(1/(1-x))\le x/(1-x)$ gives an error at most
$8/[N(1-\beta^2)]$.  The numerator correction is at most
$1/[N(1-\beta^2)]$.  This proves (ii) with $C_0=9$, and (iii) follows from the sharper bound~\eqref{eq:ht-sharp-finite}.

For the total integrated remainder, the identity, the signs, and $\chi_N\ge0$ give
\[
  0\le-\mathfrak R_N(t;\beta)
  \le1+\frac{\beta^2\rho_t}{N}.
\]
Therefore
\[
  D_N(\beta)
  \le\left(1+\frac{\beta^2}{N}\right)
       \frac1b\log\frac1{1-b}
  \le\frac{2}{1-\beta^2}.
\]
Finally, tracking the factor $\beta^2$ through
Lemma~\ref{lem:interp-pointwise-app} gives
\[
  0\le-\mathfrak R_N^{\rm interp}(t;\beta)
  \le\frac{2\beta^2\rho_t^2}{N}
       \partial_\rho\widehat\chi_N(\rho_t;\beta).
\]
Since $\theta_N^{(\beta)}\ge1-\beta^2$,
$\partial_\rho\widehat\chi_N\ge0$, and
$\widehat\chi_N(1;\beta)=\chi_{\rm SG}(N;\beta)$, integration gives
\[
  \frac{2\beta^2\chi_{\rm SG}(N;\beta)}{N(1-\beta^2)}
  \le\frac{4\beta^2}{N(1-\beta^2)^2},
\]
by (i).
\end{proof}

\begin{corollary}[Variance upper bound in the near-critical window]
\label{cor:ht-window}
If $1-\beta_N^2=cN^{-1/3}$ for fixed $c>0$, then
\[
  \Var F_N(\beta_N)
  \le\frac16\frac{\log N}{N^2}+O_c(N^{-2}).
\]
This matches the leading coefficient in the near-critical variance formula~\cite{DeyKang2026} as an upper bound only. It
does not give the lower bound or central limit theorem from that paper.
\end{corollary}

\begin{proof}
Insert $1-\beta_N^2=cN^{-1/3}$ into \eqref{eq:ht-upper} and divide by
$2N^2$.
\end{proof}

\begin{remark}[Near-critical counterparts of the critical coefficients]
\label{rem:ht-coefficient-dictionary}
Corollary~\ref{cor:ht-window} is a consequence of
\eqref{eq:ht-upper}: if $1-\beta_N^2=cN^{-\gamma}$ with $\gamma\in(0,1)$ and
$c>0$ fixed, then the logarithm contributes $\gamma\log N+O_c(1)$ and the bound has no finite-size error, so
\[
  \Var F_N(\beta_N)
  \le\frac{\gamma}{2}\,\frac{\log N}{N^2}+O_c(N^{-2}).
\]
The window exponent $\gamma$ and the upper-bound coefficient $\gamma/2$ are
therefore paired exactly as the susceptibility exponent and the variance
upper-bound coefficient are paired at exact criticality
(Theorem~\ref{thm:critical-variance-upper}). The susceptibility scale
$(1-\beta_N^2)^{-1}$ plays the role of $\chi_{\rm SG}(N)$.  In particular,
$\gamma=1/3$ is Corollary~\ref{cor:ht-window}, matching the coefficient $1/6$ in~\cite{DeyKang2026}, while $\gamma=1/2$ gives
$\Var F_N(\beta_N)\le\frac14(\log N)/N^2+O_c(N^{-2})$, the near-critical
counterpart of the coefficient proved at exact criticality in~\cite{Schertzer2026CriticalSK} and obtained there from
the classical overlap bound in
Remark~\ref{rem:dictionary-instances}.
\end{remark}

\begin{corollary}[Classical high-temperature fluctuations force remainder vanishing]
\label{cor:ht-vanishing}
For every fixed $\beta\in(0,1)$,
\[
  D_N(\beta)\longrightarrow0.
\]
\end{corollary}

\begin{proof}
Fix $0<\beta<1$ and put
\[
 V_\beta:=\frac12\left(\log\frac1{1-\beta^2}-\beta^2\right),
 \qquad
 X_N:=\log Z_N^{\rm off}(\beta)-\log\E Z_N^{\rm off}(\beta).
\]
The classical high-temperature central limit theorem
\cite{AizenmanLebowitzRuelle1987,CometsNeveu1995} gives
$X_N\Longrightarrow N(-V_\beta/2,V_\beta)$.
For an independent copy $X_N'$ it follows that
$X_N-X_N'\Longrightarrow N(0,2V_\beta)$.
Portmanteau's theorem applied to the nonnegative continuous function
$x^2$ gives
\[
 \liminf_{N\to\infty}\Var X_N
 =\frac12\liminf_{N\to\infty}\E(X_N-X_N')^2
 \ge V_\beta.
\]
The sharp finite-size bound~\eqref{eq:ht-sharp-finite} gives
$\Var X_N\le V_\beta$, hence $\Var X_N\to V_\beta$.
Lemma~\ref{lem:GOE-vs-offdiag} then gives
$2N^2\Var F_N(\beta)\to\log(1/(1-\beta^2))$.
Equation~\eqref{eq:ht-remainder-identity}, whose error tends to zero at
fixed $\beta$, forces $\beta^2D_N(\beta)\to0$.
\end{proof}

\begin{remark}[GOE normalization and the classical constant]
The spin-independent GOE diagonal adds $\beta^2/2$ to the variance of the
unnormalized log-partition function.  It converts the off-diagonal constant
$\frac12(\log(1/(1-\beta^2))-\beta^2)$ into
$\frac12\log(1/(1-\beta^2))$, exactly the leading integral in
\eqref{eq:ht-split}.
\end{remark}

\begin{remark}[Degeneration at criticality]
The bounded-remainder estimate in \eqref{eq:ht-remainder-bounds} deteriorates like
$(1-\beta^2)^{-1}$. The upper bound~\eqref{eq:ht-upper} has no
finite-size error, but diverges as $\beta\uparrow1$. In the
decomposition~\eqref{eq:ht-remainder-identity}, the stated error bound
no longer tends to zero when $N(1-\beta^2)=O(1)$. At exact criticality,
\[
  \theta_N^{(1)}(t)=1-e^{-t}+\frac{2e^{-t}}N\asymp t+\frac1N
  \qquad(0\le t\le1),
\]
so the finite-$N$ gap does not vanish at $t=0$ but becomes singular in the
thermodynamic limit.  The cutoff $t_\ast\asymp\chi_{\rm SG}(N)^{-1}$ and the
critical-window remainder replace the fixed high-temperature gap.
\end{remark}

\section{Temperature extensions}\label{app:temperature-extensions}

This appendix gives two consequences of existing critical estimates.
The first extends the entropy proof of~\cite{YimingChen2026} to a bounded
temperature window. The second propagates ordinary overlap moments
on a smaller interval. Their distinct disorder averages are kept explicit.

\subsection{A uniform variance bound}

\begin{proposition}[Extension of the entropy argument]
\label{prop:uniform-window-variance}
For every fixed $B\ge0$, there is $C_B<\infty$ such that
\begin{equation}\label{eq:uniform-window-variance}
 \sup_{0\le\beta^2\le1+BN^{-1/3}}
 N^2\Var F_N^{\rm off}(\beta)\le\frac16\log N+C_B.
\end{equation}
\end{proposition}
\begin{proof}
Let $L_{N,\beta}=Z_N^{\rm off}(\beta)/\E Z_N^{\rm off}(\beta)$ and
$\mathcal E_N(\beta)=\E[L_{N,\beta}\log L_{N,\beta}]$.
The inputs from~\cite[Section~3.1, Propositions~2--3 and Corollary~1]{YimingChen2026},
including its Gaussian convexity bound, are
\begin{align}
 N^2\Var F_N^{\rm off}(\beta)&\le2\mathcal E_N(\beta),
 \label{eq:entropy-var-extension}\\
 \mathcal E_N(1-N^{-1/3})&\le\tfrac1{12}\log N+C,
 \label{eq:entropy-base-extension}\\
 0\le\mathcal E_N'(\beta)&\le\tfrac\beta2
  \left(\E|\mathcal C_{N,p_N(\beta)}(1)|-1\right),
 \label{eq:entropy-derivative-extension}
\end{align}
where $\mathcal C_{N,p}(1)$ denotes the component of vertex 1 in $G(N,p)$,
\[
 p_N(\beta)=\E\tanh^2(\gamma+\sqrt\gamma G),
 \qquad\gamma=\beta^2/N,\qquad G\sim N(0,1).
\]
The off-diagonal parameter $\lambda$ in that paper equals our $\beta$.
The first and third estimates hold for every $\beta\ge0$.
The elementary bound $\tanh^2x\le x^2$ yields, throughout the stated range,
\[
 p_N(\beta)\le\gamma+\gamma^2
 \le N^{-1}+A_BN^{-4/3},\qquad A_B=B+(1+B)^2.
\]
The critical random graph estimate
$\E\sum_i|\mathcal C_i|^2=O_A(N^{4/3})$
of~\cite[Corollary~5.3]{JansonSpencer2007}, numbered Corollary~5.2 in
its arXiv version, where the sum runs over the components
$\mathcal C_i$ of $G(N,p)$, is equivalent to
$\E|\mathcal C_{N,p}(1)|=N^{-1}\E\sum_i|\mathcal C_i|^2=O_A(N^{1/3})$ at
$p=N^{-1}+AN^{-4/3}$ for every fixed real $A$, including $A>0$.
Monotone coupling therefore bounds the right side of
\eqref{eq:entropy-derivative-extension} by $C_BN^{1/3}$.
Put $\beta_0=1-N^{-1/3}$ and $\beta_*=(1+BN^{-1/3})^{1/2}$.
Since $\beta_*-\beta_0\le(1+B/2)N^{-1/3}$, the integral is $O_B(1)$.
Use \eqref{eq:entropy-base-extension}, monotonicity of $\mathcal E_N$,
and \eqref{eq:entropy-var-extension} to conclude. Finitely many small
$N$ are absorbed into $C_B$.
\end{proof}

The correlations behind \eqref{eq:entropy-derivative-extension} are
averaged under the planted observation law. They do not give ordinary
SK susceptibility control. Also, at fixed $\beta>1$ outside the shrinking
window, the susceptibility inequality, valid for every $\beta>0$, is not asserted to improve the known superconcentration bound~\cite{Chatterjee2009}.
For fixed $\beta>1$, the variance of the off-diagonal log-partition function
lies between $c_\beta N^{4/15}$ and $C_\beta N^{7/15}$ for all sufficiently
large $N$~\cite[Theorem~1.4]{AronowLopatto2026}.

\subsection{Comparison of ordinary overlap moments}

For $s\ge0$ write $\nu_s(f)=\E\langle f\rangle_{\beta=\sqrt{s}}$ and
$m_p(s)=\nu_s(|R_{12}|^p)$, $p>0$. These averages use the ordinary
Gaussian disorder law.

\begin{lemma}[Temperature comparison]\label{lem:temperature-comparison}
For every $p>0$ and $s,s_0\ge0$,
\begin{equation}\label{eq:temperature-moment-comparison}
 m_p(s)\le\nu_{s_0}\!\left(|R_{12}|^p
                  e^{2N|s-s_0|R_{12}^2}\right).
\end{equation}
\end{lemma}
\begin{proof}
Gaussian differentiation of the two Gibbs weights gives
\[
 m_p'(s)=\frac N2\nu_s\!\left(
 |R_{12}|^p(R_{12}^2-4R_{13}^2+3R_{34}^2)\right).
\]
Each mixed expectation is nonnegative and, by H\"older and replica
exchangeability, at most $m_{p+2}(s)$. Thus
\begin{equation}\label{eq:temperature-moment-hierarchy}
 |m_p'(s)|\le2N m_{p+2}(s).
\end{equation}
Iterating this inequality along the interval from $s_0$ to $s$ gives,
with $d=|s-s_0|$,
\[
 m_p(s)\le\sum_{j=0}^L\frac{(2Nd)^j}{j!}m_{p+2j}(s_0)
              +\frac{(2Nd)^{L+1}}{(L+1)!}.
\]
The remainder uses $|R_{12}|\le1$ and vanishes as $L\to\infty$ at
fixed $N,d$. Summing proves \eqref{eq:temperature-moment-comparison}.
\end{proof}

\begin{lemma}[A tail bound for larger critical overlaps]
\label{lem:quartic-overlap-tail}
There are $c,C>0$ such that, for sufficiently large $N$ and
$N^{-2/9}\le u\le1$,
\begin{equation}\label{eq:quartic-overlap-tail}
 \nu_1(|R_{12}|\ge u)\le C e^{-cNu^4}.
\end{equation}
\end{lemma}
\begin{proof}
For this calculation let $Z=2^{-N}Z_N^{\rm GOE}(1)$ and $\mathcal F=\log Z$.
Its Gibbs measure is the same as that of the off-diagonal model.
Put $b_N=N/4-(\log N)/12$. The free-energy tail estimate
in~\cite[Lemma~5.5]{DuHuang2026} is
\[
 \mathbb P(\mathcal F<b_N-v)
 \le2e^{-v^2/(C\log^2N)},\qquad v\ge0.
\]
On the complementary event, bound $Z^{-2}$ by $e^{-2b_N+2v}$.
The Gaussian expectation of the restricted two-replica numerator is
$e^{N/2}\E_{\rm unif}[e^{NR_{12}^2/2}\mathbf1_{\{|R_{12}|\ge u\}}]$.
The binomial entropy
\[
 I(q)=\tfrac12(1+q)\log(1+q)+\tfrac12(1-q)\log(1-q)
 \ge q^2/2+q^4/12
\]
and the at most $N+1$ overlap values give
\begin{equation}\label{eq:tail-split}
 \nu_1(|R_{12}|\ge u)
 \le N^{1/6}(N+1)e^{2v-Nu^4/12}
             +2e^{-v^2/(C\log^2N)}.
\end{equation}
Set $v=Nu^4/48$. Since $Nu^4\ge N^{1/9}\gg\log^2N$ in the stated
range, the polynomial factor and the second term are absorbed into
$Ce^{-cNu^4}$. This proves \eqref{eq:quartic-overlap-tail}.
\end{proof}

\begin{theorem}[Stability of the critical overlap moments]
\label{thm:moment-stability}
There is $\eta>0$ such that, for every fixed $p>0$, uniformly on
$s\ge0$, $|s-1|\le\eta N^{-4/9}$,
\begin{align}
 m_p(s)&\le C_pN^{-p/3},\label{eq:uniform-moment-temperature}\\
 |m_p(s)-m_p(1)|&\le C_pN^{-p/3-1/9}.
 \label{eq:moment-temperature-stability}
\end{align}
In particular,
\begin{equation}\label{eq:kurtosis-temperature-stability}
 \frac{m_4(s)}{m_2(s)^2}
 =\frac{m_4(1)}{m_2(1)^2}+O(N^{-1/9})=O(1).
\end{equation}
\end{theorem}
\begin{proof}
The critical input~\cite[Theorem~1.4]{DuHuang2026} gives
\[
 \nu_1(e^{aN^{1/3}|R_{12}|})\le2,
 \qquad m_2(1)\ge c_*N^{-2/3}
\]
for fixed $a,c_*>0$. Put $Y=|R_{12}|$, $r=N^{-2/9}$,
$d=|s-1|$, and $\vartheta=2Nd\le2\eta N^{5/9}$.
For $Y\le r$ and $\eta\le a/4$,
$\vartheta Y^2\le(a/2)N^{1/3}Y$. Absorbing the fixed power in the
remaining exponential gives
\[
 \nu_1(Y^pe^{\vartheta Y^2}\mathbf1_{\{Y\le r\}})
 \le C_pN^{-p/3}.
\]
For $u\ge r$, choose also $\eta\le c/4$, where $c$ is from
\eqref{eq:quartic-overlap-tail}. Then
$\vartheta u^2\le(c/2)Nu^4$. With $f(u)=u^pe^{\vartheta u^2}$,
integration of the tail gives
\begin{align*}
 \nu_1(f(Y)\mathbf1_{\{Y>r\}})
 &\le f(r)\nu_1(Y\ge r)+\int_r^1 f'(u)\nu_1(Y\ge u)\dd u\\
 &\le C\left[r^pe^{-cNr^4/2}
  +\int_r^1(pu^{p-1}+2\vartheta u^{p+1})e^{-cNu^4/2}\dd u\right]\\
 &\le C_p(1+\vartheta)e^{-cN^{1/9}/2}=O_p(N^{-p/3}).
\end{align*}
Lemma~\ref{lem:temperature-comparison} proves
\eqref{eq:uniform-moment-temperature}. Integrating
\eqref{eq:temperature-moment-hierarchy} and using the bound with $p+2$
gives $2N\eta N^{-4/9}C_{p+2}N^{-(p+2)/3}
=O_p(N^{-p/3-1/9})$.
The lower bound for $m_2(1)$ therefore persists on this interval.
Applying the estimates at $p=2,4$ to the ratio proves
\eqref{eq:kurtosis-temperature-stability}. Finitely many small sizes
are included by enlarging constants.
\end{proof}

The exponent $4/9$ comes from balancing the linear critical tail
$N^{1/3}u$ and the quartic tail $Nu^4$ at $u=N^{-2/9}$.
Using only the linear tail and $|R_{12}|\le1$ gives the shorter
$N^{-2/3}$ temperature interval. Neither estimate reaches a fixed
$N^{-1/3}$ window. A uniform critical bound
$\nu_1(e^{a_2N^{2/3}R_{12}^2})\le C$ would, by
\eqref{eq:temperature-moment-comparison}, control every fixed moment
on a small interval at that scale. Such a quadratic exponential moment
is not proved here and does not follow from the stated linear
exponential moment alone.

\paragraph*{Use of AI tools.}
This work originates in a research project initiated by the author well before the use of generative AI tools.
During the preparation and revision of this manuscript, the author used ChatGPT and Claude to assist with the development and checking of mathematical arguments and calculations, the preparation of code for computational checks, and the revision of the exposition.
The author takes full responsibility for the contents of the manuscript.

{
\small
\setlength{\bibsep}{4pt plus 1pt minus 1pt}

}


\begin{thebibliography}{46}
\providecommand{\natexlab}[1]{#1}
\providecommand{\url}[1]{\texttt{#1}}
\expandafter\ifx\csname urlstyle\endcsname\relax
  \providecommand{\doi}[1]{doi: #1}\else
  \providecommand{\doi}{doi: \begingroup \urlstyle{rm}\Url}\fi

\bibitem[Sherrington and Kirkpatrick(1975)]{SherringtonKirkpatrick1975}
David Sherrington and Scott Kirkpatrick.
\newblock Solvable model of a spin-glass.
\newblock \emph{Physical Review Letters}, 35\penalty0 (26):\penalty0
  1792--1796, 1975.
\newblock \doi{10.1103/PhysRevLett.35.1792}.

\bibitem[Guerra and
  Toninelli(2002{\natexlab{a}})]{GuerraToninelli2002ThermoLimit}
Francesco Guerra and Fabio~Lucio Toninelli.
\newblock The thermodynamic limit in mean field spin glass models.
\newblock \emph{Communications in Mathematical Physics}, 230\penalty0
  (1):\penalty0 71--79, 2002{\natexlab{a}}.
\newblock \doi{10.1007/s00220-002-0699-y}.

\bibitem[Parisi(1979)]{Parisi1979Infinite}
Giorgio Parisi.
\newblock Infinite number of order parameters for spin-glasses.
\newblock \emph{Physical Review Letters}, 43\penalty0 (23):\penalty0
  1754--1756, 1979.
\newblock \doi{10.1103/PhysRevLett.43.1754}.

\bibitem[Parisi(1980)]{Parisi1980OrderParameter}
Giorgio Parisi.
\newblock The order parameter for spin glasses: a function on the interval
  0--1.
\newblock \emph{Journal of Physics A: Mathematical and General}, 13\penalty0
  (3):\penalty0 1101--1112, 1980.
\newblock \doi{10.1088/0305-4470/13/3/042}.

\bibitem[Guerra(2003)]{Guerra2003}
Francesco Guerra.
\newblock Broken replica symmetry bounds in the mean field spin glass model.
\newblock \emph{Communications in Mathematical Physics}, 233\penalty0
  (1):\penalty0 1--12, 2003.
\newblock \doi{10.1007/s00220-002-0773-5}.

\bibitem[Talagrand(2006)]{Talagrand2006Parisi}
Michel Talagrand.
\newblock The {P}arisi formula.
\newblock \emph{Annals of Mathematics}, 163\penalty0 (1):\penalty0 221--263,
  2006.
\newblock \doi{10.4007/annals.2006.163.221}.

\bibitem[Talagrand(2011{\natexlab{a}})]{Talagrand2011Vol1}
Michel Talagrand.
\newblock \emph{Mean Field Models for Spin Glasses. Volume I: Basic Examples},
  volume~54 of \emph{Ergebnisse der Mathematik und ihrer Grenzgebiete. 3. Folge
  / A Series of Modern Surveys in Mathematics}.
\newblock Springer, 2011{\natexlab{a}}.
\newblock ISBN 978-3-642-15201-6.
\newblock \doi{10.1007/978-3-642-15202-3}.

\bibitem[Panchenko(2013)]{Panchenko2013Book}
Dmitry Panchenko.
\newblock \emph{The {S}herrington--{K}irkpatrick Model}.
\newblock Springer Monographs in Mathematics. Springer, 2013.
\newblock ISBN 978-1-4614-6288-0.
\newblock \doi{10.1007/978-1-4614-6289-7}.

\bibitem[Aizenman et~al.(1987)Aizenman, Lebowitz, and
  Ruelle]{AizenmanLebowitzRuelle1987}
Michael Aizenman, Joel~L. Lebowitz, and David Ruelle.
\newblock Some rigorous results on the {S}herrington--{K}irkpatrick spin glass
  model.
\newblock \emph{Communications in Mathematical Physics}, 112\penalty0
  (1):\penalty0 3--20, 1987.
\newblock \doi{10.1007/BF01217677}.

\bibitem[Comets and Neveu(1995)]{CometsNeveu1995}
Francis Comets and Jacques Neveu.
\newblock The {S}herrington--{K}irkpatrick model of spin glasses and stochastic
  calculus: The high temperature case.
\newblock \emph{Communications in Mathematical Physics}, 166\penalty0
  (3):\penalty0 549--564, 1995.
\newblock \doi{10.1007/BF02099887}.

\bibitem[Guerra and Toninelli(2002{\natexlab{b}})]{GuerraToninelli2002CLT}
Francesco Guerra and Fabio~Lucio Toninelli.
\newblock Central limit theorem for fluctuations in the high temperature region
  of the {S}herrington--{K}irkpatrick spin glass model.
\newblock \emph{Journal of Mathematical Physics}, 43\penalty0 (12):\penalty0
  6224--6237, 2002{\natexlab{b}}.
\newblock \doi{10.1063/1.1515109}.
\newblock URL \url{https://arxiv.org/abs/cond-mat/0201092}.

\bibitem[de~Almeida and Thouless(1978)]{AlmeidaThouless1978}
J.~R.~L. de~Almeida and D.~J. Thouless.
\newblock Stability of the {S}herrington--{K}irkpatrick solution of a spin
  glass model.
\newblock \emph{Journal of Physics A: Mathematical and General}, 11\penalty0
  (5):\penalty0 983--990, 1978.
\newblock \doi{10.1088/0305-4470/11/5/028}.

\bibitem[M{\'e}zard et~al.(1987)M{\'e}zard, Parisi, and
  Virasoro]{MezardParisiVirasoro1987}
Marc M{\'e}zard, Giorgio Parisi, and Miguel~A. Virasoro.
\newblock \emph{Spin Glass Theory and Beyond: An Introduction to the Replica
  Method and Its Applications}, volume~9 of \emph{World Scientific Lecture
  Notes in Physics}.
\newblock World Scientific, 1987.
\newblock ISBN 9789971501167.
\newblock \doi{10.1142/0271}.

\bibitem[Chen and Lam(2019)]{ChenLam2019}
Wei-Kuo Chen and Wai-Kit Lam.
\newblock Order of fluctuations of the free energy in the {SK} model at
  critical temperature.
\newblock \emph{ALEA, Latin American Journal of Probability and Mathematical
  Statistics}, 16:\penalty0 809--816, 2019.
\newblock \doi{10.30757/ALEA.v16-29}.

\bibitem[Collins-Woodfin and Le(2025{\natexlab{a}})]{CollinsWoodfinLe2025}
Elizabeth~W. Collins-Woodfin and Han~Gia Le.
\newblock Free energy fluctuations in {SK} and related spin glass models: A
  literature survey, 2025{\natexlab{a}}.
\newblock URL \url{https://arxiv.org/abs/2510.26960}.
\newblock arXiv:2510.26960.

\bibitem[Dey and Kang(2026)]{DeyKang2026}
Partha~S. Dey and Taegu Kang.
\newblock Fluctuations for the {S}herrington--{K}irkpatrick spin glass model
  near the critical temperature, 2026.
\newblock URL \url{https://arxiv.org/abs/2603.05636}.
\newblock arXiv:2603.05636.

\bibitem[Landon(2022)]{Landon2022SphericalCritical}
Benjamin Landon.
\newblock Free energy fluctuations of the two-spin spherical
  {SK} model at critical temperature.
\newblock \emph{Journal of Mathematical Physics}, 63\penalty0 (3):\penalty0
  033301, 2022.
\newblock \doi{10.1063/5.0054298}.

\bibitem[Aspelmeier(2008{\natexlab{a}})]{Aspelmeier2008PRL}
Timo Aspelmeier.
\newblock Free-energy fluctuations and chaos in the
  {S}herrington--{K}irkpatrick model.
\newblock \emph{Physical Review Letters}, 100:\penalty0 117205,
  2008{\natexlab{a}}.
\newblock \doi{10.1103/PhysRevLett.100.117205}.

\bibitem[Aspelmeier(2008{\natexlab{b}})]{Aspelmeier2008JStatMech}
Timo Aspelmeier.
\newblock An exact relation between free energy fluctuations and bond chaos in
  the {S}herrington--{K}irkpatrick model.
\newblock \emph{Journal of Statistical Mechanics: Theory and Experiment},
  2008\penalty0 (04):\penalty0 P04018, 2008{\natexlab{b}}.
\newblock \doi{10.1088/1742-5468/2008/04/P04018}.
\newblock URL \url{https://arxiv.org/abs/0802.3007}.

\bibitem[Schertzer(2026)]{Schertzer2026CriticalSK}
Adrien Schertzer.
\newblock The order of free energy fluctuations in the critical
  {S}herrington--{K}irkpatrick model revisited, 2026.
\newblock URL \url{https://arxiv.org/abs/2606.21360}.
\newblock Preprint, 18 pages, arXiv:2606.21360.

\bibitem[Chatterjee(2009)]{Chatterjee2009}
Sourav Chatterjee.
\newblock Disorder chaos and multiple valleys in spin glasses, 2009.
\newblock URL \url{https://arxiv.org/abs/0907.3381}.
\newblock arXiv:0907.3381; later expanded into the monograph
  \emph{Superconcentration and Related Topics} (Springer, 2014).

\bibitem[Talagrand(2011{\natexlab{b}})]{Talagrand2011Vol2}
Michel Talagrand.
\newblock \emph{Mean Field Models for Spin Glasses. Volume II: Advanced
  Replica-Symmetry and Low Temperature}, volume~55 of \emph{Ergebnisse der
  Mathematik und ihrer Grenzgebiete. 3. Folge}.
\newblock Springer, Heidelberg, 2011{\natexlab{b}}.
\newblock \doi{10.1007/978-3-642-22253-5}.

\bibitem[Du and Huang(2026)]{DuHuang2026}
Hang Du and Brice Huang.
\newblock Fluctuations of the {S}herrington--{K}irkpatrick free energy at
  critical temperature, 2026.
\newblock URL \url{https://arxiv.org/abs/2607.02172}.
\newblock arXiv:2607.02172.

\bibitem[Chen(2026)]{YimingChen2026}
Yiming Chen.
\newblock Gaussian convexity principles for sharp moderate deviations of
  Gaussian maxima and critical {SK} free energy variance, 2026.
\newblock URL \url{https://arxiv.org/abs/2607.21392}.
\newblock arXiv:2607.21392v3.

\bibitem[Du and Huang(2026)]{DuHuang2026Overlap}
Hang Du and Brice Huang.
\newblock Overlap distribution of the critical
  {Sherrington--Kirkpatrick} model, 2026.
\newblock URL \url{https://arxiv.org/abs/2608.08752}.
\newblock arXiv:2608.08752.

\bibitem[Tierz(2026)]{Tierz2026Constants}
Miguel Tierz.
\newblock In preparation, 2026.

\bibitem[Bakry et~al.(2014)Bakry, Gentil, and Ledoux]{BakryGentilLedoux2014}
Dominique Bakry, Ivan Gentil, and Michel Ledoux.
\newblock \emph{Analysis and Geometry of {M}arkov Diffusion Operators}, volume
  348 of \emph{Grundlehren der Mathematischen Wissenschaften}.
\newblock Springer, 2014.
\newblock \doi{10.1007/978-3-319-00227-9}.

\bibitem[Chatterjee(2014)]{Chatterjee2014}
Sourav Chatterjee.
\newblock \emph{Superconcentration and Related Topics}.
\newblock Springer Monographs in Mathematics. Springer, 2014.
\newblock \doi{10.1007/978-3-319-03886-5}.

\bibitem[Aspelmeier et~al.(2016)Aspelmeier, Katzgraber, Larson, Moore,
  Wittmann, and Yeo]{AspelmeierKatzgraberLarsonMooreWittmannYeo2016}
Timo Aspelmeier, Helmut~G. Katzgraber, Derek Larson, M.~A. Moore, Matthew
  Wittmann, and Joonhyun Yeo.
\newblock Finite-size critical scaling in {I}sing spin glasses in the
  mean-field regime.
\newblock \emph{Physical Review E}, 93\penalty0 (3):\penalty0 032123, 2016.
\newblock \doi{10.1103/PhysRevE.93.032123}.

\bibitem[Price(1958)]{Price1958}
R.~Price.
\newblock A useful theorem for nonlinear devices having {G}aussian inputs.
\newblock \emph{{IRE} Transactions on Information Theory}, 4\penalty0
  (2):\penalty0 69--72, 1958.
\newblock \doi{10.1109/TIT.1958.1057444}.

\bibitem[Bouchaud et~al.(2003)Bouchaud, Krzakala, and
  Martin]{BouchaudKrzakalaMartin2003}
Jean-Philippe Bouchaud, Florent Krzakala, and Olivier~C. Martin.
\newblock Energy exponents and corrections to scaling in {I}sing spin glasses.
\newblock \emph{Physical Review B}, 68\penalty0 (22):\penalty0 224404, 2003.
\newblock \doi{10.1103/PhysRevB.68.224404}.

\bibitem[Tracy and Widom(1996)]{TracyWidom1996}
Craig~A. Tracy and Harold Widom.
\newblock On orthogonal and symplectic matrix ensembles.
\newblock \emph{Communications in Mathematical Physics}, 177\penalty0
  (3):\penalty0 727--754, 1996.
\newblock \doi{10.1007/BF02099545}.

\bibitem[Baik and Lee(2016)]{BaikLee2016}
Jinho Baik and Ji~Oon Lee.
\newblock Fluctuations of the free energy of the spherical
  {S}herrington--{K}irkpatrick model.
\newblock \emph{Journal of Statistical Physics}, 165\penalty0 (2):\penalty0
  185--224, 2016.
\newblock \doi{10.1007/s10955-016-1610-0}.

\bibitem[Baik and Lee(2017)]{BaikLee2017Ferromagnetic}
Jinho Baik and Ji~Oon Lee.
\newblock Fluctuations of the free energy of the spherical
  {S}herrington--{K}irkpatrick model with ferromagnetic interaction.
\newblock \emph{Annales Henri Poincar\'e}, 18\penalty0 (6):\penalty0
  1867--1917, 2017.
\newblock \doi{10.1007/s00023-017-0562-5}.

\bibitem[Collins-Woodfin and
  Le(2025{\natexlab{b}})]{CollinsWoodfinLe2025Bipartite}
Elizabeth~W. Collins-Woodfin and Han~Gia Le.
\newblock Free energy fluctuations of the bipartite spherical {SK} model at
  critical temperature.
\newblock \emph{Annales Henri Poincar\'e}, 26\penalty0 (3):\penalty0
  1087--1147, 2025{\natexlab{b}}.
\newblock \doi{10.1007/s00023-024-01448-9}.

\bibitem[Johnstone et~al.(2024)Johnstone, Klochkov, Onatski, and
  Pavlyshyn]{JohnstoneKlochkovOnatskiPavlyshyn2024}
Iain~M. Johnstone, Yegor Klochkov, Alexei Onatski, and Damian Pavlyshyn.
\newblock Spin glass to paramagnetic transition and triple point in spherical
  {SK} model.
\newblock \emph{Journal of Statistical Physics}, 191\penalty0 (8):\penalty0 99,
  2024.
\newblock \doi{10.1007/s10955-024-03296-7}.

\bibitem[Janson and Spencer(2007)]{JansonSpencer2007}
Svante Janson and Joel Spencer.
\newblock A point process describing the component sizes in the critical window
  of the random graph evolution.
\newblock \emph{Combinatorics, Probability and Computing}, 16\penalty0 (4):\penalty0
  631--658, 2007.
\newblock \href{https://arxiv.org/abs/math/0505529}{arXiv:math/0505529}.

\bibitem[Helffer and Sj\"ostrand(1994)]{HelfferSjostrand1994}
Bernard Helffer and Johannes Sj\"ostrand.
\newblock On the correlation for {K}ac-like models in the convex case.
\newblock \emph{Journal of Statistical Physics}, 74\penalty0 (1--2):\penalty0
  349--409, 1994.
\newblock \doi{10.1007/BF02186817}.

\bibitem[Guerra and Toninelli(2002{\natexlab{c}})]{GuerraToninelli2002QRC}
Francesco Guerra and Fabio~Lucio Toninelli.
\newblock Quadratic replica coupling in the {S}herrington--{K}irkpatrick mean
  field spin glass model.
\newblock \emph{Journal of Mathematical Physics}, 43\penalty0 (7):\penalty0
  3704--3716, 2002{\natexlab{c}}.
\newblock \doi{10.1063/1.1483378}.

\bibitem[Chen et~al.(2017)Chen, Dey, and
  Panchenko]{ChenDeyPanchenko2017ExternalField}
Wei-Kuo Chen, Partha Dey, and Dmitry Panchenko.
\newblock Fluctuations of the free energy in the mixed $p$-spin models with
  external field.
\newblock \emph{Probability Theory and Related Fields}, 168\penalty0
  (1--2):\penalty0 41--53, 2017.
\newblock \doi{10.1007/s00440-016-0705-5}.

\bibitem[Barra et~al.(2015)Barra, Contucci, Mingione, and
  Tantari]{BarraContucciMingioneTantari2015}
Adriano Barra, Pierluigi Contucci, Emanuele Mingione, and Daniele Tantari.
\newblock Multi-{S}pecies mean field spin glasses. {R}igorous results.
\newblock \emph{Annales Henri Poincar\'e}, 16:\penalty0 691--708, 2015.
\newblock \doi{10.1007/s00023-014-0341-5}.

\bibitem[Panchenko(2015)]{Panchenko2015MSK}
Dmitry Panchenko.
\newblock The free energy in a multi-species {S}herrington--{K}irkpatrick
  model.
\newblock \emph{The Annals of Probability}, 43\penalty0 (6):\penalty0
  3494--3513, 2015.
\newblock \doi{10.1214/14-AOP967}.

\bibitem[Collins-Woodfin and Le(2025{\natexlab{c}})]{CollinsWoodfinLe2025MSK}
Elizabeth~W. Collins-Woodfin and Han~Gia Le.
\newblock Order of fluctuations of the free energy in the positive
  semi-definite {MSK} model at critical temperature, 2025{\natexlab{c}}.
\newblock URL \url{https://arxiv.org/abs/2501.11732}.
\newblock arXiv:2501.11732.

\bibitem[Janson(1997)]{Janson1997}
Svante Janson.
\newblock \emph{Gaussian Hilbert Spaces}, volume 129 of \emph{Cambridge Tracts
  in Mathematics}.
\newblock Cambridge University Press, 1997.
\newblock ISBN 978-0-521-56128-0.
\newblock \doi{10.1017/CBO9780511526169}.

\bibitem[Schur(1911)]{Schur1911}
I.~Schur.
\newblock Bemerkungen zur {T}heorie der beschr\"ankten {B}ilinearformen mit
  unendlich vielen {V}er\"anderlichen.
\newblock \emph{Journal f\"ur die reine und angewandte Mathematik},
  140:\penalty0 1--28, 1911.
\newblock \doi{10.1515/crll.1911.140.1}.

\bibitem[Horn and Johnson(2013)]{HornJohnson2013}
Roger~A. Horn and Charles~R. Johnson.
\newblock \emph{Matrix Analysis}.
\newblock Cambridge University Press, 2nd edition, 2013.


\bibitem[Aronow and Lopatto(2026)]{AronowLopatto2026}
P.~M. Aronow and Patrick Lopatto.
\newblock Quantitative {P}arisi formulas and fluctuations in the
  {S}herrington--{K}irkpatrick model, 2026.
\newblock \href{https://arxiv.org/abs/2609.09103}{arXiv:2609.09103}.

\end{thebibliography}
\end{document}